\documentclass[10pt,reqno]{amsart}

\usepackage[utf8]{inputenc}
\usepackage[T1]{fontenc}
\usepackage{amsmath, amssymb, amsthm}
\usepackage{graphicx} 
\usepackage[colorlinks=true,allcolors=blue]{hyperref}
\usepackage{geometry}   
\usepackage{hyperref}
\usepackage{xcolor}
\usepackage{orcidlink}
\usepackage{setspace} 
\usepackage{mathtools}
\DeclareMathOperator{\aut}{aut}% Paquete 
\DeclareMathOperator{\Aut}{Aut}
\DeclareMathOperator{\Hol}{Hol}
\DeclareMathOperator{\Iso}{Iso}
\DeclareMathOperator{\Lie}{Lie}
\DeclareMathOperator{\End}{End}
\DeclareMathOperator{\SO}{SO}
\DeclareMathOperator{\OO}{O}

\title{The distribution of symmetry of Lorentzian naturally reductive nilmanifolds}

\author{Brian Luporini}
\address{CONICET and Departamento de Matemática, ECEN-FCEIA, Universidad Nacional de Rosario, Av. Pellegrini 250, S2000BTP Rosario, Argentina.}
\email{\href{mailto:luporini@fceia.unr.edu.ar}{luporini@fceia.unr.edu.ar}}
\urladdr{\href{https://orcid.org/0009-0003-5935-0582}{0009-0003-5935-0582}}

\author{Silvio Reggiani}
\address{CONICET and Departamento de Matemática, ECEN-FCEIA, Universidad Nacional de Rosario, Av. Pellegrini 250, S2000BTP Rosario, Argentina.}
\email{\href{mailto:reggiani@fceia.unr.edu.ar}{reggiani@fceia.unr.edu.ar}}
\urladdr{\href{https://orcid.org/0000-0002-8549-5828}{0000-0002-8549-5828}}

\author{Francisco Vittone}
\address{CONICET and Departamento de Matemática, ECEN-FCEIA, Universidad Nacional de Rosario, Av. Pellegrini 250, S2000BTP Rosario, Argentina.}
\email{\href{mailto:vittone@fceia.unr.edu.ar}{vittone@fceia.unr.edu.ar}}
\urladdr{\href{https://orcid.org/0000-0002-3883-273X}{0000-0002-3883-273X}}

\date{\today}
%, Secondary ???.}

\theoremstyle{plain}
\newtheorem{theorem}{Theorem}[section]
\newtheorem{lemma}[theorem]{Lemma}
\newtheorem{corollary}[theorem]{Corollary}
\newtheorem{proposition}[theorem]{Proposition}

\theoremstyle{definition}
\newtheorem{definition}[theorem]{Definition}

\theoremstyle{remark}
\newtheorem{observation}[theorem]{Remark}

\numberwithin{equation}{section}

\newcommand{\R}{\mathbb{R}}

\newcommand{\got}{\mathfrak}

\DeclareMathOperator{\Ad}{Ad}
\DeclareMathOperator{\ad}{ad}
\DeclareMathOperator{\Id}{Id}

\begin{document}
%\fancyhf{}

\begin{abstract}
We study  $2$-step nilpotent Lorentzian Lie groups $N$, which are naturally reductive with respect to a certain class of transitive subgroups of isometries. We describe the isotropy representation and prove that its fixed points give raise to the distribution of symmetry of $N$. This generalizes some known results for the Riemannian case. 
\end{abstract}
\maketitle

\section{Introduction}

In this paper we deal with the geometry of $2$-step nilpotent Lie groups, endowed with a left-invariant Lorentzian metric which is naturally reductive with respect to a suitable presentation group.

Naturally reductive nilpotent Lie groups with left-invariant metrics have been widely studied both in the Riemannian and pseudo-Riemannian settings. 

In the Riemannian case, Wolf proved in \cite{wolf1962} that that if a connected nilpotent Lie group $N \subset \Iso(M)$ acts transitively on a differentiable manifold $M$ then $N$ is unique, it is the nilradical of the isometry group, and the transitive action of $N$ is also simple. Thus, $M$ can be identified with the nilpotent Lie group $N$ equipped with a left-invariant metric. Furthermore, the subgroup $H$ of isometries fixing the identity element coincides with the group $H^{\aut}$ of isometric automorphisms of $N$ and therefore the isometry group is the semidirect product $\Iso(M) = N \rtimes H$. Further developments on this subject were made by Kaplan in \cite{kaplan1981}, where he studied the case of $H$-type Lie groups, Wilson \cite{wilson1982} and Gordon \cite{gordon1985} among others. In particular, Gordon proved that a naturally reductive nilpotent Riemannian Lie group  with a left-invariant metric must be, at most, $2$-step nilpotent. Later, Lauret gave a description and obtained interesting geometric properties of naturally reductive nilmanifolds constructed via representations of compact Lie algebras (cf.~ \cite{lauret1998, lauret1999}).

In the pseudo-Riemannian case, Ovando gave a description of pseudo-Riemannian naturally reductive $2$-step nilpotent Lie groups with a left-invariant metric (cf.\ \cite{ovando2010}). She showed, however, that not all naturally reductive pseudo-Riemannian nilmanifolds are $2$-step nilpotent.  Moreover, in \cite{delbarco2014} del Barco and Ovando gave an example of a nilmanifold $N$ where the group $N \rtimes H$ described by Wolf in \cite{wolf1962} is smaller that $\Iso(N)$ and they gave conditions, based on the eigenvalues of the Ricci tensor, for the equality to hold. 

We are particularly interested in the Lorentzian case. In recent works, Wolf, Nikolayevsky, Chen and Zhang made a major breakthrough proving that under certain conditions a naturally reductive Lorentzian nilmanifold, with respect to the subgroup $N \rtimes H^{\aut}$ of $\Iso(N)$, must be $2$-step nilpotent as in the Riemannian case (cf.~ \cite{chen-wolf-zhang2020, nikolayevsky-wolf2022}).

In this paper, we complete the study of $2$-step nilpotent naturally reductive Lie groups $N$ with a left-invariant Lorentzian metric. In Section \ref{sec:2}, we present some general well-known results on the geometry of $N$, and characterize the existence of a flat factor in terms of some properties of the Lie algebra $\got n$ of $N$ (Theorem~\ref{lemma:injective}). In Section \ref{section:via representations}, we study those $N$ that arise via a representation $\pi: \got g \to \End(\got v)$ of a compact Lie algebra (this method was introduced by Lauret \cite{lauret1999} in the Riemannian case and generalized by Ovando \cite{ovando2010} for pseudo-Riemannian metrics). In particular, we describe the decomposition of $\pi: \got n \to \End(\got v)$ into invariant subspaces (Theorem \ref{teo:imp1}). In Section \ref{section4}, we describe the isotropy algebra $\got h^{\operatorname{aut}}$ of $N \rtimes H^{\operatorname{aut}}$ and its action on the Lie algebra $\got n = \Lie(N)$ (Theorem \ref{teo:16}). If one assumes that $\got g$ is semisimple, this description was obtained by Ovando \cite{ovando2010}. However, as we  prove in Corollary \ref{corolario 2}, this is never the case if $N$ is Lorentzian. 

The understanding of the action of the isotropy algebra is fundamental for the study of the \emph{distribution of symmetry}. Namely, if $M$ is a pseudo-Riemannian manifold, and $p\in M$, the symmetry subspace of $M$ at $p$ is defined as 
\begin{equation*}
  \got s_p = \{X_p : X \in \mathcal K_c(M) \text{ and } (\nabla X)_p = 0\},
\end{equation*}
where $\mathcal K_c(M)$ is the Lie algebra of complete Killing fields of $M$. If $M$ is homogeneous, the map $p\mapsto \got s_p$ defines an $\Iso(M)$-invariant distribution, called the distribution of symmetry of $M$.

The distribution  of symmetry a Riemannian homogeneous space was first introduced by Olmos, Tamaru and the second author in \cite{olmos2013}, and it has been widely studied in different contexts (cf.\ \cite{cardoso2024, may2021, reggiani2021, reggiani2018, berndt-olmos-reggiani2016, podesta2015}). In Section \ref{section5}, we introduce this distribution for a homogeneous pseudo-Riemannian space $M$. We prove that if it is non-degenerate, then it is integrable and its integrable manifolds are geodesically complete, homogeneous, totally geodesic, locally symmetric submanifolds of $M$ (see Lemma \ref{teo:locsym}). 

Finally, in Section \ref{section6}, we study the distribution of symmetry of a Lorentzian $2$-step nilpotent, naturally reductive Lie group with a left-invariant metric, and prove that, as in the Riemannian case (cf.\ \cite{reggiani2019}), it is given by the fixed points of the (connected) isotropy representation. 

We hope that the results presented here encourage the study of this interesting geometric invariant to the more general pseudo-Riemannian setting. 

\section{Geometry of 2-step nilpotent Lie groups}\label{sec:2}

In this section, we shall briefly recall some aspects on the geometry of 2-step nilpotent Lie groups endowed with a left-invariant metric. For more details we refer to \cite{ovando2010} and \cite{eberlein1994}.

Let $\got n$ be a 2-step nilpotent metric Lie algebra, i.e.,  $\got n$ is endowed with a non-degenerate symmetric bilinear form $\langle \cdot ,\cdot \rangle$. Assume that the center $\got z$ of $\got n$ is a non-degenerate subspace of $\got n$ and consider the orthogonal decomposition 
\begin{equation}\label{eq:descortn}
  \got n = \got z \oplus \got v
\end{equation}
with  $\got v = \got z^{\perp}$. Since $\got n$ is $2$-step nilpotent,  $[\got n, \got n] \subset \got z$. In particular, $\ad_X(\got v) \subset \got z$ for each $X \in \got v$ and so there exists a linear map $j : \got z \rightarrow \got{so}(\got v)$ such that $j(Z)(X) = (\ad_X)^*(Z)$ for each $Z \in \got z$ and $X \in \got v$ (here, $(\ad_X)^*$ is the transpose of $\ad_X$, cf.\ \cite{eberlein1994} and \cite{ovando2010}). More precisely, 
\begin{equation}\label{eq:def_map_j}
  \langle [X, Y], Z \rangle = \langle j(Z)X, Y \rangle, \qquad \text{for } X, Y \in \got v, \, Z \in \got z.
\end{equation}
From \eqref{eq:def_map_j} it follows that $\ker j = [\got n, \got n]^{\perp}$ in $\got z$ and so $j$ is injective if and only if $[\got n,\got n]=\got z$. 
Moreover, if $[\got n, \got n]$ is a non-degenerate subspace of $\got n$, then $\got z$ decomposes orthogonally as the direct sum 
\begin{equation}\label{eq:descortz}
  \got z = \ker j \oplus [\got n, \got n].
\end{equation}

\begin{observation}\label{rem:trivialsubrep}
  Even if $j$ is not injective, one has that $\cap_{Z \in \got z}\ker j(Z) = \{0\}$. In fact, if $X \in \got v$ is such that $j(Z)X = 0$ for each $Z \in \got z$ then for each $Y \in \got v$, $\langle [X, Y], Z \rangle = \langle j(Z)X, Y \rangle = 0$. Hence $[X, Y] = 0$ for every $Y \in \got v$ and so $X \in \got z$. This implies that $X = 0$. 
\end{observation}

Let $N$ be the simply connected $2$-step nilpotent Lie group whose Lie algebra is $\got n$ (i.e., we identify $\got n$ with the Lie algebra of left-invariant vector fields of $N$). Then the metric on $\got n$ induces a left-invariant metric on $N$, which we will still denote by $\langle \cdot, \cdot \rangle$. 
Denote by $\nabla$ the Levi-Civita connection of $(N, \langle \cdot, \cdot \rangle)$. Recall that $\nabla$ is left-invariant, i.e., if $U, V \in \got n$, then $\nabla_U V \in \got n$. It follows from \cite{ovando2010} that
\begin{equation}\label{eq:nabla_in_2stepnilpo}
  \begin{cases}
    \nabla_{X} Y = \frac{1}{2} [X, Y],  & \text{if } X, Y \in \got v, \\
    \nabla_XZ = \nabla_ZX = - \frac{1}{2} j(Z)X, & \text{if } X \in \got v, \, Z \in \got z, \\
    \nabla_ZZ' = 0, &\text{if } Z, Z' \in \got z.
  \end{cases}
\end{equation}

Observe that the last two equalities of \eqref{eq:nabla_in_2stepnilpo} show that, if $j$ is not injective, then every element in $\ker j$ is a parallel left-invariant vector field. Hence, in the Riemannian case, the injectivity of $j$ is equivalent to the non-existence of a de Rham flat factor of $N$ (cf.\ \cite[Proposition 2.7]{eberlein1994}). We shall see that in the pseudo-Riemannian case the injectivity of $j$ is equivalent to the non-existence of a flat factor under the additional hypothesis that the commutator $[\got n, \got n]$ is non-degenerate. Recall first the de Rahm-Wu decomposition theorem (cf.\ \cite{wu1964}).

\begin{theorem}\label{theorem:Wu}
  Let $M$ be a geodesically complete simply connected pseudo-Riemannian manifold and let $p \in M$. Let $\Hol(M, p)$ be the holonomy group of $M$ at $p$ and denote by $V_0$ the maximal subspace of $M$ on which $\Hol(M, p)$ acts trivially. Suppose that $V_0$ is non-degenerate, so $T_pM$ admits a decomposition into mutually orthogonal subspaces $T_pM = V_0 \oplus V_1$. Then $M$ is isometric to a direct product $M_0 \times  M_1$, with $M_0$  flat, $T_{p_0} M_0 = V_0$, $T_{p_1} M_1 = V_1$, where $p$ identifies with $(p_0, p_1)$, and $\Hol(M, p) \simeq \Hol(M_1, p_1)$. 
\end{theorem}

We say that $M$ has a non-trivial \emph{flat de Rham-Wu factor} if the subspace $V_0$ where the holonomy acts trivially is non-trivial and non-degenerate. Otherwise, we say that $M$ has no flat factor. The non-degenerate manifold $M_0$ in Theorem \ref{theorem:Wu} is called the flat de Rham-Wu factor, or simply the flat factor of $M$. 

If $M$ is a pseudo-Riemannian manifold and $p \in M$, the \emph{nullity subspace} of $M$ at $p$ is given by 
\begin{align}
  \nu_p & = \{v \in T_pM: R(v, w) = 0 \text{ for all } w \in T_pM\} \notag \\
  & = \{v \in T_pM: R(v, w) u = 0 \text{ for all } w, u \in T_pM\} \label{eq:nullity} \\
  & = \bigcap_{v, w \in T_pM} \ker R(v, w) \notag,
\end{align}
where $R$ is the curvature tensor of the Levi-Civita connection of $M$, i.e., \begin{equation*}
  R(X, Y)Z = \nabla_X(\nabla_Y Z) - \nabla_Y(\nabla_X Z) - \nabla_{[X, Y]} Z.
\end{equation*}   
It follows from the Ambrose-Singer Theorem that if $M_0$ is a flat factor of $M$, then 
\begin{equation}\label{eq:v0vsnu}
  V_0=T_pM_0\subset \nu_p
\end{equation} 
for all $p \in M$ (observe however that the existence of nullity do not imply, even for Riemannian homogeneous spaces,  the existence of a flat factor, cf.~\cite{DOV}). When $M = N$ is a pseudo Riemannian $2$-step nilpotent Lie group, equality holds in \eqref{eq:v0vsnu}:

\begin{lemma}\label{lemma:injective0}
  Let $N$ be a simply connected $2$-step nilpotent Lie group with a left-invariant metric $\langle \cdot, \cdot \rangle$ such that the center $\got z$ is non-degenerate. Let $\got v = \got z^{\bot}$ and $j: \got z \to \got{so}(\got v)$ be defined as \eqref{eq:def_map_j}. We identify $T_eN$ with $\got n$ in the usual way. Let $V_0 \subset T_eN \simeq \got n$ be the maximal subspace on which the holonomy group $\Hol(N, e)$ of $N$ at $e$ acts trivially, and let $\nu_e$ be the nullity subspace  of $N$ at $e$. Then 
  \begin{equation*}
    V_0=\nu_e=\ker j.
  \end{equation*}
\end{lemma}

\begin{proof}
  Let $Z \in \ker j$. From \eqref{eq:nabla_in_2stepnilpo} it follows that $\nabla_{Z'}Z = 0$ if $Z' \in \got z$ and $\nabla_X Z = -\frac{1}{2}j(Z)X = 0$ if $X\in \got v$. Hence $Z$ is a parallel vector field and so $Z \in V_0$. We conclude that $\ker j\subset V_0\subset \nu_e$. 

  Let now $W \in \nu_e$. Write $W = Z + X$ with $Z \in \got z$ and $X \in \got v$. Then for each $A, B \in \got n$,
  \begin{equation*}
    0 = R(A, B) W = R(A, B) Z + R(A, B) X.
  \end{equation*}
  From \eqref{eq:nabla_in_2stepnilpo} one easily gets (see also \cite{ovando2010} or \cite{eberlein1994}) that if $A\in \got v$ and $B\in\got z$  then 
  \begin{equation*}
    \begin{cases}
      R(A, B) Z = -\frac{1}{4}(j(B) \circ j(Z))A \in \got v,\\
      R(A, B) X = -\frac{1}{4}[A, j(B)X] \in \got z.
    \end{cases}
  \end{equation*}
  So for each $A\in \got v$ and $B\in \got z$, it follows that $R(A, B) W = 0$ if and only if $R(A, B) Z = R(A, B) X =0$. Now, if $R(A, B) Z = 0$ for each $B \in \got z$, then $j(Z) A \in \cap_{B \in \got z} \ker j(B)=\{0\}$ (see Remark \ref{rem:trivialsubrep}). So $j(Z)A = 0$ for each $A \in \got v $ and hence $Z\in \ker j$. If $R(A, B) X = 0$ for each $A \in \got v$, then $j(B)X \in \got z$ for every $B \in \got z$. But $j(B)X \in \got v$, and so $j(B)X = 0$ for each $B\in \got z$. Again from Remark \ref{rem:trivialsubrep}, we obtain that $X=0$. We conclude that $W = Z \in \ker j$, and so $\nu_e \subset \ker j$.
\end{proof}

\begin{theorem}\label{lemma:injective}
  Let $N$ be a simply connected $2$-step nilpotent Lie group with a pseudo-Riemannian left-invariant metric $\langle\cdot,\cdot\rangle$ such that the center $\got z$ of the Lie algebra $\got n$ of $N$ is a non-degenerate subspace of $\got n$. Let $\got v=\got z^{\bot}$ and $j:\got z\to \got{so}(\got v)$ be defined as \eqref{eq:def_map_j}. Then the following statements are equivalent:
  \begin{enumerate}
  \item $j$ is injective.
  \item  $[\got n,\got n]$ is non-degenerate and $N$ has no de Rham-Wu flat factor. 
  \end{enumerate}
\end{theorem}

\begin{proof}
If $j$ is inyective, then $\ker j=\{0\}$ and so $[\got n,\got n]=\got z$ is non-degenerate. Moreover, from Lemma \ref{lemma:injective0}, $\nu_e=\ker j$ is trivial, and so $N$ has no flat factor. 

Now if $N$ has no flat factor, then either $V_0=\ker j$ is degenerate or $V_0=\{0\}$. The first situation can not happen since $[\got n,\got n]$ is non-degenerate. Hence $j$ is inyective.
\end{proof}

\begin{observation}
  Clearly the hypothesis of $[\got n,\got n]$ being non-degenerate can not be dropped from Theorem \ref{lemma:injective}. In fact, a degenerate $[\got n,\got n]$ readily implies that $\ker j$ is non trivial and degenerate. Hence the subspace $V_0$ of fixed points of the holonomy group $\Phi$ of $N$ at $e$ is degenerate, and so $N$ has no  de Rham-Wu flat factor, even though $j$ is not injective.  
\end{observation}

Denote by $\Iso(N)$ the full isometry group of $N$ and let $H = \Iso(N)_e$ be the isotropy group at the identity element $e\in N$. We have that 
\begin{equation*}
  \Iso(N) = L_N \cdot H
\end{equation*}
where $L_N\simeq N$ is the subgroup of $\Iso(N)$ consisting of the left-translations. Observe that $H\cap L_N = \{ \operatorname{Id} \}$.

Consider the Lie subgroup $H^{\operatorname{aut}}$ of $H$ consisting of the isometric automorphism of $N$, i.e., 
\begin{equation*}\label{eq:haut}
  H^{\operatorname{aut}}=\Aut(N)\cap \Iso(N)=\Aut(N)\cap H
\end{equation*}
and the Lie subgroup $\Iso^{\operatorname{aut}}(N)$ of $\Iso(N)$ given by 
\begin{equation*}
  \Iso^{\aut}(N) = L_N \cdot H^{\operatorname{aut}}.
\end{equation*}
It is standard to see that  $L_N$ is a normal subgroup of $\Iso(N)^{\operatorname{aut}}$ and hence (cf.\ \cite{delbarco2014})
\begin{equation}\label{eq:2}
  \Iso^{\operatorname{aut}}(N) = L_N  \rtimes H^{\operatorname{aut}}\simeq N  \rtimes H^{\operatorname{aut}.}
\end{equation}

Since $N$ is simply connected, $\Aut(N) \simeq \Aut(\got n)$. Therefore 
\begin{equation*}
    H^{\operatorname{aut}} \simeq \OO(\got n) \cap \Aut(\got n),
\end{equation*}
where $\OO(\got n)$ is the orthogonal group of $\got n$ with respect to the given metric. With these identifications, the Lie algebra of $\Iso^{\operatorname{aut}} (N)$ is $\got{iso}^{\operatorname{aut}} (N) \simeq \got n \rtimes \got h^{\operatorname{aut}}$ where
\begin{equation}\label{eq:haut_lie}
    \got h^{\operatorname{aut}} = \operatorname{Der}(\got n) \cap \got{so}(\got n)
\end{equation}
is the Lie algebra of skew-symmetric derivations of $\got n$.  

Recall that under the identification $\got{iso}^{\aut}(N) \simeq \got n \rtimes \got h^{\aut}$, if $U, V \in \got n$ and $A, B \in \got h^{\aut}$, the Lie bracket of $\got{iso}^{\aut}(N)$ is given by
\begin{align}\label{eq:corcheteiso}
  [U, V]_{\got{iso}^{\aut}(N)} = [U, V]_{\got n}, && [A, B]_{\got{iso}^{\aut}(N)} = [A, B]_{\got h^{\aut}}, && [A, U]_{\got{iso}^{\aut}(N)}=A(U).
\end{align}

\begin{observation} 
    If the metric on $\got n$ is positive definite (i.e. the left-invariant metric induced on $N$ is Riemannian), then $\Iso(N)=\Iso^{\operatorname{aut}}(N)$ (cf.~ \cite{wolf1962}). This is no longer true for a pseudo-Riemannian nilmanifold (cf.~ \cite{delbarco2014}).
\end{observation}

We are interested in characterizing when $N$ is naturally reductive with respect to the presentation group $\Iso^{\aut}(N)$. 

Let $M=G/H$ be a pseudo-Riemannian homogeneous space $M=G/H$, with $G$ a Lie subgroup of $\Iso(M)$ and $H=G_e$, the isotropy at the identity $e$. Let $\got g$ and $\got h$ be the Lie algebras of $G$ and $H$ respectively. Recall that $M$ is  \emph{naturally reductive} with respect to $G$ if there exists a subspace $\got m$ of $\got g$ such that \begin{equation}\label{eq:reddesc}\got g = \got m \oplus \got h\ \text{ with }\ \ad_{\got g}(\got h) \got m \subset \got m\end{equation} and  for every $U,V,W \in \got m$, 
\begin{equation}\label{eq:DefNatRed}
    \langle [U, V]_{\got m}, W \rangle + \langle V, [U,W]_{\got m} \rangle = 0,
\end{equation}
where $[\cdot, \cdot]_{\got m}$ denotes the $\got m$-component of the Lie bracket in $\got g$. 

If $N$ is a $2$-step nilpotent simply connected Lie group with a left-invariant pseudo-Riemannian metric, then $\Iso^{\aut}(N)$ acts transitively on $N$, since it contains all left-translations, and $\Iso^{\aut}(N)_e=H^{\aut}$. Hence 
\begin{equation*}
    N=\Iso^{\aut}(N)/H^{\aut}
\end{equation*}
is a pseudo-Riemannian homogeneous space. 

One can characterize when $N$ is naturally reductive with respect to the transitive group $\Iso^{\operatorname{aut}}(N)$ in terms of the map $j:\got z\to \got{so}(\got v)$:
\begin{lemma}[\cite{ovando2010}]\label{lemma:2-stepNatRedu}
    If $N$ is  naturally reductive for the group $\Iso^{\operatorname{aut}}(N)$ then $j(\got z)$ is a subalgebra of $\got{so}(\got v)$ and for every $Z\in \got z$, there exists an element $\tau_Z\in \got{so}(\got z)$ such that 
    \begin{equation}\label{eq:reljtau}
        [j(Z),j(Z')] = j(\tau_Z(Z')), \quad Z'\in \got z.
    \end{equation}
    If $j$ is injective, then the converse holds. 
\end{lemma}

\begin{observation}\label{rem:2pasosnrj}
  Under the hypothesis of Lemma \ref{lemma:2-stepNatRedu} it follows that one can define a Lie bracket $[\cdot,\cdot]_{\got z}$ on $\got z$ by putting 
  \begin{equation*}
      [Z,Z']_{\got z} = \tau_Z(Z')
  \end{equation*}
  (where $\tau_Z$ is defined by \eqref{eq:reljtau}) and $[\cdot,\cdot]_{\got z}$ is such that $j:\got z \rightarrow \got{so}(\got v)$ is a representation of the Lie algebra $(\got z,[\cdot,\cdot]_{\got z})$. In addition, $j:\got z\to \got{so}(\got v)$ has no trivial subrepresentations, i.e. $\cap_{Z\in \got z} \ker j(Z)=\{0\}$ (cf. Remark \ref{rem:trivialsubrep}). Moreover, since $\tau_Z\in \got{so}(\got z)$ for each $Z\in \got z$, one gets that if $\langle \cdot,\cdot \rangle_{\got z}$ denotes the restriction to $\got z$ of the metric on $\got n$, then $\langle\cdot,\cdot\rangle_{\got z}$ is $\ad$-invariant with respect to the Lie bracket $[\cdot,\cdot]_{\got z}$. 
\end{observation}

\section{Naturally reductive Lorentzian $2$-step nilpotent Lie groups via representations}\label{section:via representations}

In this section we shall recall the construction of a $2$-step nilpotent Lie algebra $\got n$ from a representation $\pi:\got g\to \End(\got v)$, where $\got g$ is a Lie algebra with a particular inner product and $\got v$ is a real vector space, such that the associated 2-step nilpotent simply connected Lie group $N$ is naturally reductive with respect to the presentation group $\Iso^{\aut}(N)$.  In addition, we will present some interesting properties when the metric resulting metric on $\got n$ is is Lorentzian. 

\begin{definition}[cf.~ \cite{lauret1999,ovando2010}]
    A \emph{data set} is a triplet $(\got g, \got v, \pi)$ where:
 \begin{enumerate}
     \item $\got g$ is a Lie algebra endowed with an ad-invariant metric $\langle \cdot, \cdot \rangle_{\got g}$, i.e. $\ad_Z \in \got{so}(\got g,\langle \cdot, \cdot \rangle_{\got g})$ for each $Z \in \got g$;
     \item $\got v$ is a real vector space;
     \item $\pi: \got g \rightarrow \text{End}(\got v)$ is a real faithful representation  without trivial subrepresentations,  i.e. $\cap_{Z \in \got g} \ker \pi(Z) = 0$;
     \item $\got v$ is endowed with a $\pi(\got g)$-invariant inner product $\langle \cdot, \cdot \rangle_{\got v}$, i.e., $\pi:\got g\to \got{so}(\got v)$.
 \end{enumerate}
\end{definition}

Given a data set $(\got g,\got v,\got \pi)$ define
\begin{equation*}
    \got n = \got g \oplus \got v
\end{equation*}
and consider a  metric on $\got n$ setting
\begin{equation}\label{eq:metricNR}
    \langle \cdot, \cdot \rangle|_{\got g \times \got g} = \langle \cdot, \cdot \rangle_{\got g}, \ \ \ \langle \cdot, \cdot \rangle|_{\got v \times \got v} = \langle \cdot, \cdot \rangle_{\got v}, \ \ \ \langle \got g, \got v \rangle =0.
\end{equation}
One can define a  Lie bracket on $\got n$ by 
\begin{equation}\label{eq:bracketNR}
    \left\{\begin{array}{l}
         [\got g,\got n] = 0, \ [\got v,\got v]\subset \got g, \\
         \langle [X,Y], Z \rangle = \langle \pi(Z)X, Y \rangle \quad \text{for } Z \in \got g, \, X,Y \in \got v, \\
    \end{array}\right.
\end{equation}
see \cite{lauret1999,ovando2010}. Then $\got n$ is a 2-step nilpotent metric Lie algebra which we shall denote by $\got n(\got g,\got v,\pi)$. It is immediate that the center $\got z$ of $\got n$ contains $\got g$. From equations \eqref{eq:metricNR} and \eqref{eq:bracketNR} one gets that if $X\in \got v$ belongs to $\got z$, then $X \in\cap_{Z\in \got g}\ker \pi(Z)$ and hence $X = 0$. So $\got z=\got g$ and hence the center of $\got n(\got g, \got v,\pi)$  is non-degenerate. Denote by $N(\got g,\got v,\pi)$ the simply connected 2-step nilpotent Lie group associated to $\got n(\got g,\got v,\pi)$. 
In this case, the map $j:\got z\to \got{so}(\got v)$ defined in \eqref{eq:def_map_j}  coincides with the inyective representation $\pi$ and the map $\tau_Z$, for each $Z\in \got z$, defined in Lemma \ref{lemma:2-stepNatRedu} is given by $\tau_Z=j^{-1}\circ {\ad_{j(Z)}}\circ j$. 

Hence the converse of Lemma \ref{lemma:2-stepNatRedu} holds and so $N:=N(\got g,\got v,\pi)$ with the induced left-invariant metric is a naturally reductive pseudo-Riemannian space, with respect to the presentation group $\Iso^{\operatorname{aut}}(N)$. 

Moreover, from Remark \ref{rem:2pasosnrj} we conclude that if $\got n$ is a 2-step nilpotent metric Lie algebra with non-degenerate center $\got z$ and injective $j$ then $((\got z,[\cdot,\cdot]_{\got z}),\got v=\got z^{\bot},j)$ is a data set and the associated Lie group $N$ is $N(\got z,\got v,j)$. Recall that one can guaranty an inyective $j$ if $N$ has no flat de Rham-Wu factor and $[\got n,\got n]$ is non-degenerate (Theorem \ref{lemma:injective}). 

We shall now obtain some properties of data sets $(\got g,\got v,\pi)$ such that the associated Lie group $N(\got g,\got v,\pi)$ is Lorentzian. 

\begin{definition}
  We say that a data set $(\got g,\got v,\pi)$ is a \emph{Lorentzian data set} if the metric $\langle\cdot,\cdot\rangle$ defined on $\got n=\got n(\got g,\got v,\pi)=\got g\oplus\got v$  by \eqref{eq:metricNR} has signature one. In this case, the group $N=N(\got g,\got v,\pi)$ with the left-invariant metric induced by $\langle\cdot,\cdot\rangle$ is a Lorentzian manifold. We say that $\got g$ (resp.\ $\got v$) is Riemannian if $\langle \cdot,\cdot\rangle_{\got g}$ (resp.\ $\langle\cdot,\cdot\rangle_{\got v}$) is positive definite and Lorentzian if it has signature one. 
\end{definition}

Given a Lorentizan data set $(\got g,\got n,\pi)$, since the decomposition $\got n=\got g\oplus \got v$ is orthogonal, then one of the spaces $\got g$ and $\got v$ is Riemannian and the other is Lorentzian.

\begin{proposition}\label{prop:gcompact}
   Let $(\got g, \got v, \pi)$ be a Lorentzian data set. Then $\got g$ is a compact Lie algebra. Hence $\got g=\overline{\got g}\oplus \got c$, where $\got c$ is the center of $\got g$ and $\overline{\got g}=[\got g,\got g]$ is semisimple. 
\end{proposition}

\begin{proof}
  If $\got g$ is Riemannian and $\got v$ is Lorentzian, $\got g$ is compact since the metric on $\got g$ is Riemannian and $\ad$-invariant. If $\got g$ is Lorentzian, since the representation $\pi$ is  faithful, then $\got g$ is isomorphic to $\pi(\got g)\subset \got{so}(\got v)$.  So $\got g$ is isomorphic to a subalgebra of a compact Lie algebra and hence it is compact. 
\end{proof}

Let $(\got g, \got v, \pi)$ be a Lorentzian data set. If $\got g$ is Lorentzian and  $\got v$ is Riemannian, the proof of the following result is analogous to the Riemannian case (cf.~ \cite[Lemma 3.11]{lauret1999}).

\begin{theorem}\label{lema:tec}
  Let $(\got g, \got v, \pi)$ be a data set with $\got v$ Riemannian and let $\got g = \overline{\got g} \oplus \got c$, with $\overline{\got g} = [\got g, \got g]$ and $\got c$ is the center of $\got g$. Then $\got v$ admits an orthogonal decomposition
  \begin{equation}\label{eq:descgotv}
    \got v = \got v_1 \oplus \cdots \oplus \got v_k
  \end{equation}
  into $\pi(\got g)$-irreducible subspaces, such that for each $i=1, \ldots,k$ there exist a skew-symmetric map $J_i:\got v_i \rightarrow \got v_i$ satisfying $J_i^2 = - I$ such that for every $Z\in \got c$, 
  \begin{equation*}
    \pi(Z)|_{\got v_i} = \lambda_i(Z) J_i \, \text{ for some } \lambda_i(Z)\in \R.
  \end{equation*}
\end{theorem}

Whenever $\got g$ is Riemannian and $\got v$ is Lorentzian, it is not possible to  decompose $\got v$ into $\pi(\got g)$-irreducible orthogonal subspaces, but we shall prove that one can decompose $\got v$ into an orthogonal sum of a first reducible factor, which is a $2$-dimensional Lorentzian subspace generated by  two invariant lightlike vectors, and the sum of irreducible Riemannian subspaces (cf.\ Theorem \ref{teo:imp1} below). In order to do so we first need to prove some technical results on the Lie algebra $\got{so}(1,n)$ of the Lorentzian isometry group.   

Recall that if $\got v$ is Lorentzian, say of dimension $n+1$, then $\got v$ can be identified with the Lorentzian space $\mathbb R^{1,n}$, i.e. the vector space $\mathbb R^{n+1}$ with the canonical Lorentzian metric given by 
\begin{equation*}
  \langle x,y\rangle_1=-x_1y_1+\sum_{j=2}^{n+1}x_jy_j= x^tMy, \text{ with } M=\left(\begin{array}{c|c}
  -1& 0\\
  \hline
  0 & \text{Id}_n\end{array}\right),
\end{equation*}
and $\got{so}(\got v,\langle\cdot,\cdot\rangle_{\got v})\simeq \got{so}(1,n)$, where $\got{so}(1,n)$ is the Lie algebra of the isometry group $\OO(1,n)$ of $(\mathbb R^{1,n},\langle\cdot,\cdot\rangle_1$), i.e., 
\begin{align*}
  \got{so}(1,n) & = \{ A\in \got{gl}(n+1,\mathbb R) : \langle Av,w \rangle_1+\langle v, Aw \rangle_1=0, \text{ for all } v, w \in \R^{n+1}\} \\ 
  & =\left\{\left(\begin{array}{c|c}
  0& x^t\\
  \hline
  x & B \end{array}\right) : x\in \mathbb R^n, \, B \in \got{so}(n)\right\}.
\end{align*} 

\begin{lemma}\label{lema:SO_fixedvector}
    Let $K$ be a compact subgroup of the Lie group $\SO_+(1,n)$ (the connected component of the identity in $\OO(1,n)$). Then there is a timelike vector $v$ of $\mathbb R^{1,n}$ which is fixed by all the elements of $K$.
\end{lemma}

\begin{proof}
  Consider the $n$-dimensional hyperbolic space $\mathbb H^n$, as the $n$-dimensional Riemannian submanifold of $\mathbb R^{1,n}$ given by 
  \begin{equation*}
      \mathbb H^n=\{x\in \mathbb R^{1,n}\,:\,\langle x,x\rangle=-1, \ x_1>0\}.
  \end{equation*}
  Then $\SO_+(1,n)$ is the connected component of the identity of $\Iso(\mathbb H^n)$. Hence $K\subset \SO_{+}(n,1)$ is a compact group which acts on $\mathbb H^{n}$ by isometries. Since $\mathbb H^{n}$ is complete, simply connected and has negative sectional curvature, by Cartan's Fixed Point Theorem (cf.~ \cite[Theorem 1.4.6]{eberlein1996}) $K$ has a fixed point in $\mathbb{H}^n$, which is a timelike vector of the Lorentzian space $\mathbb{R}^{1,n}$.
\end{proof}

\begin{lemma}\label{lemma:2}
  Let $\got s$ be a compact semisimple Lie subalgebra of $\got{so}(1,n)$. Then 
  \begin{equation*}
      \got v_0 = \bigcap\limits_{A \in \got s} \ker A
  \end{equation*}
  contains at least one timelike vector. In particular, $\got v_0$ is non-degenerate and if it has dimension greater than or equal to $2$, it is a Lorentzian space.
\end{lemma}

\begin{proof}
  Let $G$ be a connected subgroup of $\SO_+(1,n)$ with Lie algebra $\got s$. Since $\got s$ is compact and semisimple, $G$ is compact. By Lemma \ref{lema:SO_fixedvector}, there is a timelike vector $v\in \R^{1,n}$ such that if $A \in \got s$ then $e^{tA}(v) = v$ for every $t$. We then have that $A \cdot v = 0$ for every $A \in \got s$. That is, $v \in \got v_0$ and therefore $\got v_0$ is either a one-dimensional  subspace generated by $v$ or it is a Lorentzian subspace of $\got v$.
\end{proof}

The following result is immediate from the previous lemma.

\begin{corollary}\label{corolario 1}
  Let $\got s$ be a compact semisimple Lie algebra. Then there are no faithful representations $\rho: \got s \rightarrow \got{so}(1,n)$ without trivial subrepresntations.
\end{corollary}

\begin{corollary} \label{corolario 2}
  Let $(\got g,\got v, \pi)$ be a Lorentzian data set. Then $\got g$ is not semisimple (i.e., $\got c\neq\{0\}$). 
\end{corollary}

\begin{proof}
  Suppose $\got g$ is semisimple and decompose $\got g$ as the direct sum $\got g=\got h_1\oplus\cdots\oplus \got h_n$ of simple ideals.  From Proposition \ref{prop:gcompact}, $\got g$ is compact and hence each $\got h_i$ is simple and compact and $\dim \got h_i\geq 3$ for each $i=1,\ldots,n$. On the other hand, it is standard to see that $\got h_i \perp \got h_j$, if $i\neq j$, with respect to the $\ad$-invariant metric $\langle\cdot,\cdot\rangle_{\got g}$, and $\langle\cdot,\cdot\rangle_{\got g}$ decomposes as 
  \begin{equation*}
      \langle\cdot,\cdot\rangle_g=\lambda_1B_1+\cdots+\lambda_nB_n,
  \end{equation*}
  where $B_i$ is the Killing form of $\got h_i$ (see for example \cite{conti2024}). Since $\got h_i$ is compact, each $B_i$ is negative definite and so either $\got g$ is Riemannian or $\langle\cdot,\cdot\rangle_{\got g}$ has signature $\nu\geq 2$, which can not occur. 

  We conclude that $\got g$ must be Riemannian and hence $\got v$ is Lorentzian. But then $\pi:\got g\to \got{so}(\got v)\simeq \got{so}(1,n)$ is a faithful representation without trivial subrepresentations, which contradicts Corollary \ref{corolario 1}. 
\end{proof}

\begin{lemma}\label{lema:veryimpor}
  Let $\got a$ be an abelian subalgebra of $\got{so}(1,n)$ such that $\cap_{x\in \got a} \ker x = \{ 0 \}$. Then 
  \begin{equation*}
      \R^{1,n} = \got v_0\oplus\got v_1 \oplus \cdots \oplus \got v_l
  \end{equation*}
  is the orthogonal sum of $\got a$-invariant subspaces, such that $\got v_i$ is Riemannian and irreducible for $i\geq 1$ and $\got v_0$ is Lorentzian of dimension $2$, which is in turn the sum of two invariant (and irreducible) subspaces of dimension $1$ generated by lightlike vectors.
\end{lemma}

\begin{proof}
  We will make induction on $n$. If $n=1$, $\got a = \got{so}(1,1)$ and  $\got v_0=\mathbb R^{1,1}=\mathbb R\cdot(1,1)\oplus\mathbb R\cdot (-1,1)$, so the lemma is proved. Suppose that $n\geq 2$ and the lemma is valid for each $k < n$. Let $\mathcal A$ be the abelian connected Lie subgroup of $\SO_+(n,1)$ with Lie algebra $\got a$. 

  Since $\SO_+(1,n)$ is not abelian, we have that $\mathcal A \subsetneqq \SO_+(n,1)$. From \cite[Theorem 1.1]{diScala-olmos2001}, there are no connected proper subgroups of $\SO_+(1,n)$ which act irreducibly on $\mathbb R^{1,n}$. Therefore the action of $\mathcal A$ leaves  invariant a subspace $V_1$ of $\R^{1,n}$. If $V_1$ is Lorentzian (or Riemannian, in which case $V_1^{\perp}$ is Lorentzian and invariant) we apply the inductive hypothesis together with Theorem \ref{lema:tec} and the lemma is proved.

  Suppose then that $V_1$ is an $\mathcal A$-invariant degenerate subspace of $\R^{1,n}$. In that case, $V_1$ contains a unique lightlike direction, say $\R w_0$, and since $\mathcal A$ acts by isometries and $V_1$ is $\mathcal A$-invariant, $\mathcal A \cdot w_0 \subset \R w_0$. That is, $w_0$ is a common eigenvector of all the elements of $\mathcal A$. 

  There should exist at least one isometry $T \in \mathcal A$ such that $T(w_0) = \lambda w_0$ with $\lambda \neq \pm 1$. In fact, for each $A \in \got a$ there exists a differentiable function $\lambda_A : \R \to \R$ such that 
  \begin{equation*}
      e^{tA}(w_0)=\lambda_A(t)w_0.
  \end{equation*}
  Observe that since $e^{tA}$ is invertible, $\lambda_A(t)\neq 0$ for each $A\in \got a$ and each $t\in \R$. Since $\lambda_A(0)=1$ then $\lambda_A(t)>0$ for each $A \in \got a$ and each $t \in \R$. If we had $\lambda_A \equiv 1$ for each $A\in \got a$, then $e^{tA}w_0 = w_0$ for each $A\in \got a$ and each $t\in \R$ and so $w_0 \in \cap_{A \in \got a} \ker A = \{ 0 \}$, which is a contradiction. Then there exists $T = e^{t_0A_0}$ for some $A_0 \in \got a$ such that $Tw_0 = \lambda w_0$ with $\lambda > 0$ and $\lambda \neq 1$. In particular, $\lambda \neq \pm 1$.

  Since $\lambda \neq \pm1$, $T$ must have a second lightlike eigenvector, say $w_1$, with eigenvalue $1/\lambda$ (cf.~ \cite[Lemma 1.61]{javaloyes-sanchez2010}). 
  Let 
  \begin{equation*}
      \got v_0 = \text{span}\{ w_0, w_1 \}.
  \end{equation*}
  Then $\got v_0$ is a Lorentzian space (cf.~ \cite[Lemma 1.44]{javaloyes-sanchez2010}). We will see that $\got v_0$ is $\mathcal A$-invariant.

  Since $\got v_0$ is Lorentzian of dimension $2$, $U=\got v_0^{\bot}$ is Riemannian of dimension $n-1$ and 
  \begin{equation*}
      \R^{1,n} = \got v_0 \oplus U.
  \end{equation*}
  Let $E_{1/\lambda}$ be the eigenspace of $T$ associated with the eigenvalue $1/\lambda$. Let $w \in E_{1/\lambda}$ and write
  \begin{equation*}
      w = aw_0 + bw_1 + u
  \end{equation*}
  where $u\in U$ is a spacelike vector and $a,b\in \R$. 
  Then on the one hand
  \begin{equation*}
    Tw = \lambda a w_0 + \dfrac{b}{\lambda}w_1 + Tu
  \end{equation*}
  and on the other hand, since $w\in E_{1/\lambda}$, \begin{equation*} Tw = \dfrac{1}{\lambda} w =  \dfrac{a}{\lambda} w_0 + \dfrac{b}{\lambda}w_1 + \dfrac{1}{\lambda} u.
  \end{equation*}
  It follows that $\lambda^2a = a$ and $Tu = (1/\lambda )u$. Since $\lambda\neq \pm1$ we must have $a = 0$. On the other hand, either $u=0$ or $u$ is a spacelike eigenvector of $T$ in $E_{1/\lambda}$. But non lightlike eigenvectors of $T$ must be associated to eigenvalues $\pm 1$ (cf.~ \cite[Prop. 1.57]{javaloyes-sanchez2010}). We conclude that $u = 0$ and therefore $E_{1/\lambda} = \R w_1$. 

  Since all the isometries of $\mathcal A$ commute with $T$, they preserve its eigenspaces and therefore $\mathcal A(\R w_1) \subset \R w_1$. We conclude that $\got v_0$ is $\mathcal A$-invariant as we wanted to see. This together with Lemma \ref{lema:tec} concludes the proof.
\end{proof}

Now we can generalize Theorem \ref{lema:tec} to the case where $(\got g,\got v,\pi)$ is a data set with  Lorentzian $\got v$.

\begin{theorem}\label{teo:imp1}
  Let $(\got g,\got v,\pi)$ be a Lorentzian data set. If $\got v$ is Lorentzian, then:
  \begin{enumerate}
    \item \label{item1teoimp1} $\got v$ decomposes as an orthogonal sum 
    \begin{equation*}
        \got v = \got v_0 \oplus \got v_1 \oplus \cdots \oplus \got v_k
    \end{equation*}
    of $\pi(\got g)$-invariant subspaces, where $\got v_i$ is Riemannian and irreducible for $i\geq 1$ and $\got v_0$ is Lorentzian of dimension $2$, which is in turn the sum of two invariant (and irreducible) subspaces of dimension $1$ generated by lightlike vectors. 
    
    \item\label{item4ateoimp1} For every $i=1,\ldots,k$ there exists a skew-symmetric map $J_i:\got v_i \rightarrow \got v_i$ such that $J_i^2 = - \Id$ and for every $Z\in \got c$,
    \begin{equation*}
        \pi(Z)|_{\got v_i} = \lambda_i(Z) J_i \quad \text{ for some } \lambda_i(Z)\in \R.
    \end{equation*}

    \item \label{item5teoimp1} There exists a map $J_0\in \mathfrak{so}(\mathfrak v_0) \simeq \mathfrak{so}(1,1)$, such that $J_0^2 = \Id$ and for every $Z\in \got c$, 
    \begin{equation*}
        \pi(Z)|_{\got v_0} = \lambda_0(Z) J_0 \quad \text{ for some } \lambda_0(Z)\in \R.
    \end{equation*}
  \end{enumerate}
\end{theorem}

\begin{proof}
  From Proposition \ref{prop:gcompact}, $\got g$ is a compact subalgebra of $\got{so}(\got v)\simeq \got{so}(1,n)$ and hence 
  \begin{equation*}
      \got g = \got c \oplus [\got g, \got g]
  \end{equation*}
  where $\got c$ is the center of $\got g$, and $[\got g,\got g]$ is compact and semisimple. Form Corollary \ref{corolario 2}, $\got c \neq 0$. Let 
  \begin{equation*}
      \got u_0 = \bigcap_{Z \in [\got g, \got g]} \ker \pi(Z) \subset \got v.
  \end{equation*}
  Then $\got u_0$ is a $\pi(\got c)$-invariant subspace, since the elements of $\got c$ commute with each element of $[\got g,\got g]$, and hence $\got u_0$ is $\pi(\got g)$-invariant.
  
  Note that since the representation $\pi$ does not admit trivial subrepresentations,  $\dim(\got u_0)\geq 2$. Indeed, by Lemma \ref{lemma:2}, $\got u_0$ contains at least one timelike vector, say $X_0$. If $\dim(\got u_0) = 1$, then $\got u_0=\mathbb Rv_0$ and since $\got u_0$ is $\pi(\got c)$-invariant, it should be $\pi(Z)(X_0) = \lambda_Z X_0$ for each $Z \in \got c$. But since $\pi(\got c) \subset \got{so}(1, n)$, then $\pi(Z)(X_0)$ is orthogonal to $X_0$ and so $\lambda_Z  = 0$ for each $Z\in \got c$. So $X_0 \in \cap_{Z\in \got g} \ker \pi(Z)$, which cannot happen. Therefore, $\dim(\got v_0)=k \geq 2$, and since it contains a timelike vector, it is a Lorentzian space (cf.~ \cite[Proposition 1.44]{javaloyes-sanchez2010}). Hence
  \begin{equation*}
      \got v=\got u_0 \oplus\got u_0^{\bot}
  \end{equation*}
  and $\got u_0^{\bot}$ is Riemannian.
  
  Consider the (possibly non faithfull) representation $\mu:\got c\to \got{so}(\got u_0)\simeq \got{so}(1,k-1)$ such that $\mu(Z)(X)=\pi(Z)(X)$ for each $Z\in \got c$, $X\in \got u_0$ (i.e., $\mu(Z)$ is obtained by restricting the domain and codomain of $\pi(Z)$ to $\got u_0$). Then $\tilde{\got c}:=\mu(\got c)$ is an abelian subalgebra of $\got{so}(1, k-1)$. Observe that $\cap_{z \in \tilde{\got c}}\ker z = \cap_{Z\in \got c} \ker \mu(Z)=\{0\}$. Indeed, if $X\in \got u_0$ and  $\mu(Z)(X)=0$ for each $Z\in \got c$, then $\pi(Z)(X)=0$ and hence $X\in \cap_{Z\in \got g} \ker\pi(Z) = \{0\}$. In particular, $\tilde{\got c}\neq\{0\}$.
  
  From Lemma \ref{lema:veryimpor}, $\got u_0$ is the sum 
  \begin{equation*}
      \got u_0 = \got v_0 \oplus \cdots \oplus \got v_l
  \end{equation*}
  of $\tilde{\got c}$-invariant subspaces, such that $\got v_i$ is Riemannian and irreducible for $i \geq 1$ and $\got v_0$ is Lorentzian of dimension 2, which is in turn the sum of two invariant and irreducible subspaces of dimension 1 generated by lightlike vectors. Since for each $Z\in \got c$, $\pi(Z)(\got v_i)=\mu(Z)(\got v_i)\subset \got v_i$, and for every $Z\in [\got g, \got g]$, $\pi(Z) ( \got v_i) = \{ 0 \}$, the spaces $\got v_i$ are $\pi(\got g)$-invariant subspaces for every $i=0,...,l$.
  
  On the other hand, since $\got u_0^{\perp}$ is Riemannian, it can be decomposed as a sum 
  \begin{equation*}
      \got u_0^{\perp} = \got v_{l+1}\oplus \cdots \oplus \got v_k
  \end{equation*}
  of $\pi(\got g)$-invariant and irreducible subspaces. This concludes the proof of item \eqref{item1teoimp1}.
  
  The proof of item \eqref{item4ateoimp1} follows in the same way as in the Riemannian case (see \cite[Lemma 3.11]{lauret1999}). In order to prove item \eqref{item5teoimp1}, observe first that since $\pi([\got g,\got g])(\got v_0)=0$ then $\pi(\got c)(\got v_0)\neq 0$, otherwise we would have $\cap_{Z \in \got g} \ker\pi(Z) \neq \{0\}$. If $J_0:\got v_0\to \got v_0$ is the linear map that interchanges an orthonormal basis of $\got v_0$, then $J_0^2 = \Id$ and $\got{so}(\got v_0) = \mathbb R\cdot J_0$. Then for each $Z\in \got c$ there exists some $\lambda_0(Z)\in \mathbb R$ such that  $\pi(Z)|_{\got v_0} = \lambda_0(Z) J_0$. 
\end{proof}

We shall prove next that the kernel of any of the maps $\pi(Z)$ for $Z\in \got g$ can be decomposed accordingly to de decompositions of $\got v$ given by Theorems \ref{lema:tec} and \ref{teo:imp1}, and as a consequence that $\ker\pi(Z)$ is always a non-degenerate subspace of $\got v$.

\begin{lemma}\label{lema:ker_nodeg}
  Let $(\got g,\got v,\pi)$ be a Lorentzian data set. Decompose 
  \begin{equation*}
      \got v=\got v_0\oplus\got v_1\oplus\cdots\oplus\got v_k
  \end{equation*}
  into $\pi(\got g)$-invariant irreducible subspaces, with $\got v_0=\{0\}$ if $\got v$ is Riemannian, or $\got v_0$ a Lorentzian $2$-dimensional subspace of $\got v$ if $\got v$ is Lorentzian. Fix $Z \in \got g$ and set $\got b_i = (\ker \pi(Z)) \cap \got v_i$ and $\got w_i$ the orthogonal complement of $\got b_i$ in $\got v_i$. Then 
  \begin{equation*}
      \ker\pi(Z)=\got b_0 \oplus \got b_1 \oplus \cdots\oplus \got b_k
  \end{equation*}
  with $\got b_0 = \{0\}$ or $\got b_0 = \got v_0$. In particular, $\ker\pi(Z)$ is non-degenerate, it is Lorentzian if and only if $\got v_0 \subset  \ker\pi(Z)$, and
  \begin{equation*}
      \ker\pi(Z)^{\perp}=\got w_0\oplus \got w_1\oplus\cdots\oplus\got w_k.
  \end{equation*}
\end{lemma}

\begin{proof} 
  Let $X \in \ker \pi(Z)$ and decompose $X = X_0 + X_1 + \cdots + X_k$ with $X_i \in \got v_i$. Then $\pi(Z)(X) = 0$ if and only if $\sum \pi(Z)(X_i) = 0$, and since $\got v_i$ are $\pi(Z)$-invariant subspaces of $\got v$, we get that $\pi(Z)(X_i) = 0$ for each $i=0,\ldots,k$. So $X\in \got b_0 \oplus \got b_1 \oplus \cdots \oplus \got b_k$. The other inclusion is immediate. 

  If $\got v$ is Riemannian, the proof is complete. Suppose $\got v$ is Lorentzian, so $\got v_0 \neq \{0\}$. Observe that $\got b_i = \ker(\pi(Z)|_{\got v_i})$. Since $\pi(\got g)$ acts (perhaps non faithfully) on $\got v_0$ as $\got{so}(1,1)$, then either $\ker (\pi(Z)|_{\got v_0}) = \{ 0 \}$ or $\ker (\pi(Z)|_{\got v_0}) = \got v_0$. In any case, $\ker \pi(Z)$ is non-degenerate. The last assertion follows immediately. 
\end{proof}

\begin{observation}\label{rem:lema_ker}
    Observe that the subspaces $\got b_i$ or $\got w_i$ in the decomposition of $\ker\pi(Z)$ and $(\ker\pi(Z))^{\perp}$ given in Lemma \ref{lema:ker_nodeg} are not necessarily $\pi(\got g)$-invariant. 
    
    However, if $Z\in \got c$, the center of $(\got g, [\cdot, \cdot]_{\got g})$, then for each $Z'\in \got g$, $\pi(Z')$ commutes with $\pi(Z)$ and so $\pi(Z')$ leaves $\ker\pi(Z)$ invariant. As a consequence, $\got b_i = \got v_i\cap (\ker\pi(Z))$ is a $\pi(\got g)$-invariant subspace of $\got v_i$.  Since for each $i=1,\ldots,k$, $\got v_i$ is irreducible with respect to the action of $\pi(\got g)$, then  either $\got b_i = \{0\}$, and in consequence $\got w_i = \got v_i$, or $\got b_i = \got v_i$ and $\got w_i = \{0\}$. 
\end{observation}

\section{The isotropy algebra $\got{h}^{\operatorname{aut}}$} \label{section4}

Let $(\got g, \got v, \pi)$ be a Lorentzian data set and let $N=N(\got g, \got v,\pi)$ be the simply connected $2$-step nilpotent Lorentzian Lie group with Lie algebra $\got n=\got n(\got g,\got v,\pi)=\got g\oplus \got v$. The Lie algebra of the Lie group $\Iso^{\operatorname{aut}}(N) = N \rtimes H^{\operatorname{aut}}$, is given by
\begin{equation*}
  \got{iso}^{\aut}(N) = \got n \rtimes \got h^{\operatorname{aut}}.
\end{equation*}
So, in order to obtain $\Iso^{\aut}(N)$, one only need to compute $H^{\aut}$. It was proved in \cite{delbarco2014, ovando2010} that
\begin{equation}\label{eq:H_aut}
  H^{\aut} = \{ (\phi, T) \in \OO(\got g, \langle \cdot, \cdot \rangle_{\got g}) \times \OO(\got v, \langle \cdot, \cdot \rangle_{\got v}) : \pi(\phi Z) = T \pi(Z) T^{-1} \text{ for every } Z \in \got g\},
\end{equation}
and that its Lie algebra is
\begin{equation}\label{eq:H_aut_lie}
  \got{h}^{\operatorname{aut}} =  \{ (A,B)\in \got{so}(\got g,\langle \cdot, \cdot \rangle_{\got g}) \times \got{so}(\got v,\langle \cdot, \cdot \rangle_{\got v}) : [B,\pi(Z)] = \pi(AZ) \text{ for every } Z\in \got g\}.
\end{equation}

In this section, we will give a simpler description of $\got h^{\operatorname{aut}}$ which is analogous to that of the Riemannian case proved in \cite[Theorem 3.12]{lauret1999}. In \cite{ovando2010}, such a description was given  pseudo-Riemannian spaces under the assumption that $\got g$ is semisimple (cf. the discussion after \cite[Proposition 3.5]{ovando2010}), but as we have observed in Corollary \ref{corolario 1} this is never the case when $\got n$ is Lorentzian. 

Given a data set $(\got g,\got v,\pi)$ denote by $\End_{\pi}(\got v)$ the set of intertwining endomorphisms of $\got v$ with respect to $\pi$, that is $B \in \End(\got v)$ is in $\End_{\pi}(\got v)$ if 
\begin{equation*}
    \pi(Z)B(X) = B(\pi(Z)X)
\end{equation*}
for every $Z \in \got g$ and every $X \in \got v$, i.e., $[B, \pi(Z)] = 0$ for every $Z \in \got g$. Then:

\begin{theorem}\label{teo:16}
  Let $(\got g, \got v, \pi)$ be a Lorentzian data set. Decompose $\got g = \overline{\got g} \oplus \got c$ where $\overline{\got g}=[\got g,\got g]$ is compact semisimple and $\got c$ is the center of $\got g$. Then 
  \begin{equation*}
    \got h^{\aut} = \overline{\got g} \oplus \got u, \qquad [\overline{\got g}, \got u] = 0,
  \end{equation*}
  where $\got u = \End_{\pi}(\got v) \cap \got{so}(\got v) = \{B \in \got{so}(\got v): [B, \pi(Z)] = 0 \text{ for every } Z \in \got g\}$, and $\overline{\got g}$  acts on $\got n = \got n(\got g, \got v, \pi) = \got g \oplus \got v$ as $(\ad(Z), \pi(Z))$ for every $Z \in \overline{\got g}$.
\end{theorem}

\begin{proof}
  For simplicity, throughout this proof we will write $\got h=\got h^{\operatorname{aut}}$. We reserve the notation $[\cdot,\cdot]$ for the usual Lie bracket in $\got h\subset \End(\got n)$ and denote by $[\cdot,\cdot]_{\got g}$ the Lie bracket in $\got g$ and by $[\cdot,\cdot]_{\got n}$ the Lie bracket in $\got n$ defined by~\eqref{eq:bracketNR}. 

  Recall that $\got h$ is the Lie algebra of skew-symmetric derivations of $(\got n,[\cdot,\cdot]_{\got n})$ (cf. Equation \eqref{eq:haut_lie}) and that $\got g$ is the center of $(\got n,[\cdot,\cdot]_{\got n})$. So, if $D\in \got h$, then $D$ preserves $\got g$ and its orthogonal complement $\got v$.

  Suppose that $D = (A,B)$, with 
  \begin{equation*}
      A\in \got{so}(\got g,\langle\cdot,\cdot\rangle_\got g), \ \ \ B\in \got{so}(\got v,\langle\cdot,\cdot\rangle_\got v)
  \end{equation*}
  (cf. Equation \eqref{eq:H_aut_lie}). With the same argument as in the proof of \cite[Theorem 3.12]{lauret1999}, one can prove that $A$ is a derivation of $(\got g,[\cdot,\cdot]_{\got g})$ (cf. also \cite[Proposition 3.5]{ovando2010}). Therefore, the commutator $\overline{\got g}$ and the center $\got c$ of $\got g$  are $A$-invariant subspaces and, since $\overline{\got g}$ is semisimple, there exists an element $Z_0\in \overline{\got g}$ such that 
  \begin{equation*}
      A|_{\overline{\got g}} = \ad(Z_0)|_{\overline{\got g}}.
  \end{equation*}

  On the other hand, also following \cite{lauret1999}, one has that $(\ad(Z_0),\pi(Z_0))$ is a skew-symmetric derivation of $\got n$, i.e., $(\ad(Z_0), \pi(Z_0))\in \got h$. Hence 
  \begin{equation*}
      (A',B') = (A - \ad(Z_0), B-\pi(Z_0))
  \end{equation*}
  is an element of $\got h$ that satisfies
  \begin{equation}\label{eq:A'}
    A'|_{\overline{\got g}} = 0\ \text{ and }\ A'\got c \subset \got c.
  \end{equation}
  Let us prove that $A'|_{\got c} = 0$. This together with \eqref{eq:A'} will imply that $A'=0$.  

  Let $0 \neq Z \in \got c$. Recall that from Lemma \ref{lema:ker_nodeg}, $\ker\pi(Z)$ is a non-degenerate subspace of $\got v$. So $\got v$ decomposes orthogonally as 
  \begin{equation*}
      \got v=\ker\pi(Z) \oplus (\ker\pi(Z))^{\perp}.
  \end{equation*}
  We shall prove first that $\pi(A'Z)|_{\ker\pi(Z)}=0$. From \eqref{eq:H_aut_lie}, we have that
  \begin{equation}\label{eq:CasoRieman}
    B'\circ \pi(Z) - \pi(Z)\circ B' = \pi(A'Z).
  \end{equation}
  So if $X,Y \in \ker \pi(Z)$, we have 
  \begin{equation*}
      \langle \pi(A'Z)X, Y \rangle_{\got v} =  \langle B'(\pi(Z)X) - \pi(Z)(B'X), Y \rangle_{\got v} = \langle -\pi(Z)(B'X), Y \rangle_{\got v} = \langle B'X, \pi(Z)Y \rangle_{\got v} = 0.
  \end{equation*}
  Since $\ker\pi(Z)$ is non-degenerate, this implies that $\pi(A'Z) \equiv 0$ in $\ker \pi(Z)$ as we wanted to see.

  Let us see now that $\pi(A'Z)|_{(\ker\pi(Z))^{\perp}} = 0$. Consider the orthogonal decomposition of $\got v$ into $\pi(\got g)$-invariant subspaces given by Theorem \ref{lema:tec} if $\got v$ is Riemannian and Theorem \ref{teo:imp1} if $\got v$ is Lorentzian, i.e., 
  \begin{equation*}
      \got v = \got v_0\oplus\got v_1\oplus\cdots\oplus\got v_k
  \end{equation*}
  where $\got v_0 = \{0\}$ if $\got v$ is Riemannian and $\got v_0$ is Lorentzian of dimension $2$ if $\got v$ is Lorentzian. 

  Let $I=\{ i\in\{0,\ldots,k\}\,:\,(\ker\pi(Z))\cap \got v_i = \{ 0 \} \}$. Then
  from Remark \ref{rem:lema_ker} we have that 
  \begin{equation*}
      (\ker\pi(Z))^{\perp}=\bigoplus_{i\in I}\got v_i.
  \end{equation*}

  Fix $i\in I$. Then from Theorems \ref{lema:tec} and \ref{teo:imp1}, there exists a non-singular endomorphism $J_i\in \got{so}(\got v_i)$  and a function $\lambda_i: \got c \to  \R$ such that $\pi|_{\got c}= \lambda_i J_i$ ($J_i$ actually verifies $J_i^{-1} = -J_i$ if $i\neq 0$ and $J_0^{-1} = J_0$). Since $\pi(Z)|_{\got v_i}\neq 0$, then $\lambda_i(Z) \neq 0$. Let $K_i = \lambda_i(Z)J_i$ and $\alpha_i = \lambda_i(A'Z)/\lambda_i(Z)$. Then $K_i\in \got{so}(\got v_i)$ is a non-singular endomorphism of $\got v_i$ such that   
  \begin{equation*}
      \pi(Z)|_{\got v_i} = K_i \ \ \text{and} \ \ \pi(A'Z)|_{\got v_i}=\alpha_iK_i.
  \end{equation*}
  Define $B_i' = p_i \circ B'|_{\got v_i}: \got v_i \rightarrow \got v_i$, where $p_i$ denotes the orthogonal projection of $\got v$ onto $\got v_i$. Then $B_i'\in \got{so}(\got v_i)$ and from \eqref{eq:CasoRieman}, we have $B'_i K_i - K_i B_i' = \alpha_i K_i.$ So, 
  \begin{equation*}
    K_i^{-1} B'_i K_i - B_i' = \alpha_i\Id.
  \end{equation*} 
  Since $B_i',\,K_i'\in \got{so}(\got v_i)$, the left hand in the above equation is an element of $\got{so}(\got v_i)$ and so $\alpha_i=0$. We get $\pi(A'Z)|_{(\ker\pi(Z))^{\perp}} = 0$ as we wanted to see. 

  So we have that  $\pi(A'Z) = 0$ for each $Z \in \got c$, and since $\pi$ is faithful we conclude that $A'|_{\got c}=0$ and so $A'=0$. Hence, every element $D=(A,B)$ of $\got h$ is the form 
  \begin{equation*}
      D = (\ad(Z_0), \pi(Z_0)) + (0,B')
  \end{equation*}
  where $Z_0 \in \overline{\got g}$ and $B' = B - \pi(Z_0) \in \End_{\pi}(\got v) \cap \got{so}(\got v) = \got u$. Finally, observe that since $\pi$ is a faithful representation, $\varphi: \overline{\got g}\to \got h$ given by
  \begin{equation}\label{eq:monophi}
    \varphi(Z) = (\ad(Z), \pi(Z))
  \end{equation} 
  is a Lie algebra monomorphism and so $\overline{\got g}$ identifies with the Lie subalgebra $\varphi(\overline{\got g})=\{ (\ad(Z), \pi(Z)) : Z \in \overline{\got g}\}$ of $\got h$. We can also identify $\got u$ with $\{(0,B) : B \in \End_\pi(\got v) \cap \got{so}(\got v)\}$ and it follows that, with these identifications, $\got h = \overline{\got g} \oplus \got u$. From the definition of $\got u$, it is immediate that $\overline{\got g}$ commutes with $\got u$. Therefore $\got h = \overline{\got g} \oplus \got u$ as a sum of ideals.
\end{proof}

\begin{corollary}\label{cor:H0}
  Let $(\got g, \got v, \pi)$ be a Lorentzian data set. Decompose $\got g = \overline{\got g} \oplus \got c$, where $\overline{\got g}=[\got g,\got g]$ is compact semisimple and $\got c$ is the center of $\got g$. Then the identity component of $H^{\operatorname{aut}}$ is 
  \begin{equation*}
      (H^{\operatorname{aut}} )_0= G \times U_0
  \end{equation*}
  where 
  $U= \End_{\pi}(\got v)\cap \OO(\got v,\langle \cdot, \cdot \rangle_{\got v}) $, $G = \overline{G}/\ker \pi$ and $\overline{G}$ is the simply connected Lie group with Lie algebra $\overline{\got g}$. The group $U$ acts trivially on $\got g$ and if we also denote by $\pi$ the corresponding representation of $G$ on $\got v$, then each $g \in G$ acts on $\got n = \got g \oplus \got v$ by $(\Ad(g), \pi(g))$.
\end{corollary}

Our proof follows similar ideas as in \cite[Theorem 3.12]{lauret1999} and we include it to make the exposition self-contained.

\begin{proof}
  Let $\overline G$ be the simply connected Lie group whose Lie algebra is $\overline{\got g} $. Then $\overline G$ is compact and semisimple and if $\tilde G=\overline G\times \R^n$, where $n=\dim \got c$, then $\tilde G$ is the simply connected Lie group whose Lie algebra is $\got g$. 

  There exists a representation $\tilde \pi : \tilde G\to \OO(\got v)$ such that $d\tilde \pi_e=\pi$. Then for each $g\in \tilde G$ and each $Z \in \got g$, one has that 
  \begin{equation}\label{eq:Adpi}
    \pi(\Ad^{\tilde G}(g)(Z)) = \Ad^{\OO(\got v)}(\tilde\pi(g))(\pi(Z)) = \tilde \pi(g)\pi(Z)\tilde\pi(g)^{-1}.
  \end{equation}
  Since the metric $\langle\cdot,\cdot\rangle_{\got g}$ is $\ad$-invariant, then $\Ad^{\tilde G}(g)\in \OO(\got g,\langle\cdot,\cdot\rangle_{\got g})$. Then from equations \eqref{eq:H_aut} and \eqref{eq:Adpi} one gets that, in particular, 
  \begin{equation*}
      (\Ad^{\tilde G}(g),\tilde \pi(g))\in H^{\aut}
  \end{equation*}
  for each $g\in \overline G$. 

  Hence one has a well-defined homomorphism 
  \begin{equation*}
      \overline{\varphi}:\overline{G} \to H^{\operatorname{aut}}, \qquad g \mapsto (\Ad^{\tilde G}(g),\tilde{\pi}(g)).
  \end{equation*}
  Observe that $d\overline \varphi_e=\varphi$, where $\varphi:\overline{\got g}\to \got h^{\aut}$ is the monomorphism defined by \eqref{eq:monophi}. So $\ker(\overline\varphi)$ is a discrete subgroup of $\overline G$ (and hence a finite subgroup, since $\overline G$ is compact). Then $G=\overline{\varphi}(\overline G)$ is a compact connected subgroup of $H^{\aut}$, isomorphic to $\overline G/\ker(\tilde{\pi})$, whose Lie algebra is $\varphi(\overline {\got g})\simeq \overline{\got g}$. 

On the other hand, if $U= \text{End}_{\pi}(\got v)\cap \OO(\got v,\langle \cdot, \cdot \rangle_{\got v}) $, then the Lie algebra of $U$ is $\got u=\End_{\pi}(\got v)\cap \got{so}(\got v)$. It then follows from Theorem \ref{teo:16} that $H^{\aut}_0=G\times U_0$. 
\end{proof}

\section{The index of symmetry}\label{section5}

In this section we shall apply our results to study the distribution of symmetry of Lorentzian $2$-step nilpotent, naturally reductive Lie groups. We begin by introducing some basic definitions and properties that, to our knowledge, have only been established for Riemannian homogeneous spaces (cf.~ \cite{olmos2013}).

Let $M$ be a pseudo-Riemannian manifold. Recall that a vector field $\overline U \in \got X(M)$ is called a \emph{Killing vector field} if its flow $\{\varphi_t\}$ is given by local isometries. Equivalently, $\overline{U} \in \got{X}(M)$ is a Killing vector field if and only if for each $q\in M$ the map 
\begin{equation*}
    (\nabla \overline{U})_q:T_qM\to T_q M, \quad v\mapsto (\nabla_v \overline{U})_q
\end{equation*}
defines an element of $\got{so}(T_q M)$. Each Killing field $\overline{U} \in \mathcal K(M)$ is completely determined (if $M$ is connected) by its initial conditions $(\overline{U}_q, (\nabla \overline{U})_q)$ at any point $q\in M$. If $(\nabla \overline{U})_q=0$, then $X$ is called a \emph{transvection} at $q$. 

We denote by $\mathcal K(M)$ the Lie algebra of Killing fields of $M$ 
and by $\mathcal K_c(M)$ the Lie subalgebra of  complete Killing vector fields of  $M$. 
Then $\mathcal K_c(M)$ can be identified with $\got{iso}(M)$, the Lie algebra of $\Iso(M)$. More precisely, let $\exp:\got{iso}(M)\to \Iso(M)$ be the exponential map of the isometry group of $M$. Then the map $\Phi : \got{iso}(M) \to \mathcal K_c(M)$ defined by
\begin{equation}\label{eq:isokilling}
  \Phi(U)_q = \left.\frac{d}{dt}\right|_0\exp(tU)(q),
\end{equation}
is a Lie algebra anti-isomorphism, i.e., $\Phi([U,V])=-[\Phi(U), \Phi(V)]$ for every $U,V \in \got{iso}(M)$.  Observe that for each $U\in \got{iso}(M)$, then the flow $\{\varphi_t\}$ of $\Phi(U)$ is given by 
\begin{equation}\label{eq:flowx}
    \varphi_t(q)=\exp(tU)(q). 
\end{equation}
From now on we will denote $\tilde U:=\Phi(U)$ (observe that any complete Killing field of $M$ is $\tilde U$ for some $U\in \got{iso}(M)$).

Denote by $\got{iso}(M)_q$ the Lie algebra of the isotropy group $\Iso(M)_q$ of $\Iso(M)$ at $q$. Observe that $U\in \got{iso}(M)_q$ if and only if $\tilde U_q=0$. The transvections at a point $q\in M$ form a subspace $\tilde{\got p}^q$ of $\mathcal K_c(M)$ called the \emph{Cartan subspace} at $q$. Namely, 
\begin{equation*}
    \tilde{\got p}^q \coloneqq \{ \tilde U \in \mathcal K_c(M): (\nabla \tilde U)_q = 0 \}.
\end{equation*}
Observe that if $\tilde U,\tilde V\in \got p^q$, then $[\tilde U,\tilde V]_q = (\nabla_{\tilde U} \tilde V)_q - (\nabla_{\tilde V}\tilde U)_q = 0$. So $[\tilde U,\tilde V]\in \Phi(\got{iso}(M)_q)$. The \emph{symmetric isotropy algebra} $\tilde{\got h}^q$ at $q$ is defined by
\begin{equation*}
    \tilde{\got h}^q \coloneqq \text{span}_{\R}\{ [\tilde U,\tilde V]: \tilde U, \tilde V \in \got p^q \} = [\tilde{\got p}^q, \tilde{\got p}^q]\subset \Phi(\got{iso}(M)_q).
\end{equation*}
Then one has the direct sum (of vector spaces) 
\begin{equation}\label{eq:gq}
    \tilde{\got g}^q := \tilde{\got h}^q \oplus \tilde{\got p}^q.
\end{equation}
It is standard to prove that $[\tilde{\got h}^q, \tilde{\got h}^q] \subset \tilde{\got h}^q$ and $[\tilde{\got p}^q, \tilde{\got h}^q] \subset \tilde{\got h}^q$, so the vector space $\tilde{\got g}^q$ is a Lie subalgebra of $\Phi(\got{iso}(M))$. Denote by $\got p^q:=\Phi^{-1}(\tilde{\got p}^q)$, $\got h^q:=\Phi^{-1}(\tilde{\got h}^q)$ and $\got g^q:=\Phi^{-1}(\tilde{\got g}^q)$. Then $\got h^q$ is a Lie subalgebra of $\got{iso}_q(M)$, and $\got g^{q}=\got h^q\oplus \got p^q$ is a Lie subalgebra of $\got{iso}(M)$. 

\begin{observation}\label{rem:gx}
  The Lie algebras $\got g^q$ and $\tilde{\got g}^q$ depends on $q$. Now let $f\in \Iso(M)$ and let $x=f(q)$. Then we have that $f_*(\tilde U)\in \mathcal K_c(M)$ for every $\tilde U\in \mathcal K_c(M)$, and that $\tilde U\in \tilde{\got{p}}^q$ if and only if $f_*(\tilde U)\in \tilde{\got{p}}^x$. So, if $\tilde W=[\tilde U,\tilde V]\in \tilde{\got h}^q$ with $\tilde U,\tilde V\in \tilde{\got p}^q$, it follows that $f_*(\tilde W)=[f_*(\tilde U),f_*(\tilde V)]\in [\tilde{\got p}^x,\tilde{\got p}^x]=\tilde{\got h}^x$. Hence $\tilde{\got g}^x=f_*(\tilde{\got{g}}^q)$. 
  
  Let $G^q$ is the connected subgroup of $\Iso(M)$ whose Lie algebra is $\got g^q$, then if $x=f(q)$ for some $f\in \Iso(M)$ one gets that  $G^x=fG^qf^{-1}$. In particular,  if  $x\in G^q\cdot q$ then $G^x=G^q$. 
\end{observation}

\begin{lemma}\label{lemma:paralleltransport}
  Let $M$ be a pseudo-Riemannian manifold and let $\tilde{U} \in \mathcal K_c(M)$ with flow $\{\varphi_t\}$. Let $q\in M$ and $c(t) = \varphi_t(q)$. Denote by $\tau_t : T_qM \to T_{c(t)}M$ the parallel displacement along $c(t)$. Then:
 \begin{enumerate}
     \item\label{item1lemma} $\tau_t=(d\varphi_t)_q \circ e^{-t(\nabla \tilde{U})_q}$, where $e:\got{so}(T_qM)\to \OO(T_qM)$ is the usual exponential map;
     \item\label{item2lemma} if $\tilde{U}\in \got{p}^q$ then $c(t)$ is a geodesic of $M$. 
 \end{enumerate}
\end{lemma}

\begin{proof}
  From \cite[Remark 2.3]{olmos95}, one has that 
  \begin{equation*}
      \tau_t = (d\varphi_t)_q\circ e ^{-A_{\tilde U}},
  \end{equation*}
  where for $v\in T_qM$ and $V_t = (d\varphi_t)_q(v)$, $A_{\tilde U}(v)=\frac{D}{dt}\big|_0 V_t$ (here $\frac{D}{dt}$ represents the covariant derivative along $c(t)$). Let $\alpha(s)$ be a curve in $M$ such that $\alpha(0) = q$ and $\alpha'(0) = v$. Then 
  \begin{equation*}
      A_{\tilde{U}} (v) = \frac{D}{dt}\Big|_0 V_t = \frac{D}{dt}\Big|_0 \frac{\partial}{\partial s}\Big|_0\varphi_t(\alpha(s)) = \frac{D}{ds}\Big|_0\frac{\partial}{\partial t}\Big|_0 \varphi_t(\alpha(s)) = \frac{D}{ds}\Big|_0 \tilde{U}_{\alpha(s)} = (\nabla_v \tilde{U})_q
  \end{equation*}
  and item \ref{item1lemma} follows. Observe that $c'(t) = (d\varphi_t)_q(c'(0))$. Hence if $(\nabla \tilde{U})_p = 0$, from item \ref{item1lemma} we have that $c'(t) = \tau_t(c'(0))$ and so $c(t)$ is a geodesic. 
\end{proof}

\begin{theorem}\label{teo:locsym}
  Let $M$ be a pseudo-Riemannian manifold, $q\in M$ and let $G^q$ be the connected Lie subgroup of $\Iso(M)$ whose Lie algebra is $\got g^q$ defined by \eqref{eq:gq}. Let
  \begin{equation*}
      L(q)=G^q\cdot q
  \end{equation*}
  be the orbit of $q$ by the action of $G^q$. If $L(q)$ is a pseudo-Riemannian submanifold of $M$, then it is a geodesically complete, (homogeneous) totally geodesic, locally symmetric submanifold of $M$ and $T_xL(q)=\{U_x\,:\,U\in\got{p}^x\}$ for each $x\in L(q)$. 
\end{theorem}

\begin{proof}
  $L(q)$ is clearly homogeneous. Let $x \in L(q)$ and let $u \in T_xL(q)$. It follows from Remark \ref{rem:gx} that $G^q \cdot q = G^x \cdot x$. So there exists $\tilde W \in \tilde{\got g}^x$ such that $\tilde W_x=u$. Decompose $\tilde W=\tilde V+\tilde U$ with $\tilde V\in \tilde{\got h}^x$ and $\tilde U\in \tilde{\got p}^x$. Since $\tilde V_x=0$, we conclude that there exists $\tilde U\in \got p^x$ such that $\tilde U_x=u$. From Lemma \ref{lemma:paralleltransport}, the curve $c(t)=\varphi_t(x)$ is a geodesic such that $c(0)=x$ and $c'(0)=u$, where $\{\varphi_t\}$ is the flow of $\tilde U$. From \eqref{eq:flowx}, $c(t)\in G^x\cdot x=L(q)$ for each $t$. So $L(q)$ is totally geodesic. Since $\tilde U\in \tilde{\got p}^x$ is complete, it follows that $L(q)$ is geodesically complete. 

  Let now $\overline\nabla$ and $\overline R$ be the Levi-Civita connection and the curvature tensor of $L(q)$, respectively. Since $L(q)$ is totally geodesic, $\overline \nabla = \nabla\big|_{TL(q)^2}$ and  for each $x \in L(q)$, $\overline R_x = R_x\big|_{T_xL(q)^3}$, where $R$ is the curvature tensor of $M$. 

  Let $w,u_1,u_2,u_3\in T_xL(q)$ and let $\tilde W,\tilde U_1,\tilde U_2,\tilde U_3\in \tilde{\got p}^x$ such that $\tilde W(x)=w$, $\tilde U_i(x)=u_i$ for $i=1,2,3$. Observe that the flow of the Killing field $\tilde W$ preserves $L(q)$, and so it is a Killing field of $L(q)$. Then $\mathcal L_{\tilde W}\overline R=\mathcal L_{\tilde W}R=0$. Now
  \begin{equation*}
    0=(\mathcal L_{\tilde W} \overline{R})(\tilde U_1,\tilde U_2,\tilde U_3) = (\overline{\nabla}_{\tilde W} \overline{R})(\tilde U_1,\tilde U_2,\tilde U_3) - \overline{R}(\overline{\nabla}_{\tilde W}\tilde U_1,\tilde U_2)\tilde U_3 - \overline{R}(\tilde U_1,\overline{\nabla}_{\tilde W}\tilde U_2)\tilde U_3-\overline{R}(\tilde U_1,\tilde U_2)\overline\nabla_{\tilde W}\tilde U_3.
  \end{equation*}
  Since $\tilde U_i\in \tilde{\got p}^x$, evaluating at $x$ we have 
  \begin{equation*}
    0= (\overline{\nabla}_{\tilde W} \overline{R})_x(\tilde U_1,\tilde U_2,\tilde U_3).
  \end{equation*}
  Therefore, $({\overline\nabla} \,{\overline R})_x=0$ and so $L(q)$ is a locally symmetric space.
\end{proof}

\begin{definition}
  Let $M$ be a pseudo-Riemannian manifold and let $q\in M$. The subspace
  \begin{equation}
    \got s_q = \{ \tilde U_q : \tilde U\in \tilde{\got p}^q \} =\tilde{ \got g}^q  \cdot q \subset T_qM,
  \end{equation}
  is called the \emph{symmetry subspace} of $M$ at $q$. The dimension $i_{\got s}(q)=\dim(\got s_q)$ is called the \emph{index of symmetry} of $M$ at $q$. If $q \mapsto i_{\got s}(q)$ is constant on $M$ we call this number the \emph{index of symmetry of $M$} and we denote it by $i_{\got s}(M)$. 
\end{definition}

\begin{observation}
  If $M$ is a Riemannian manifold then $L(q)$ is a symmetric space for each $q\in M$. Hence $M$ is a symmetric space if only if $i_{\got s}(q)=\dim M=i_{\got s}(M)$ for each $q\in M$ (cf.~ \cite{olmos2013}). Informally, the index of symmetry tells us how far is a Riemannian manifold from a symmetric space. If $M$ is pseudo-Riemannian, from Theorem~\ref{teo:locsym} we have that if $i_{\got s}(M)=\dim M$ then  $M$ is a locally symmetric pseudo-Riemannian space.
\end{observation}

Suppose now that $M$ is a pseudo-Riemannian $G$-homogeneous manifold, i.e., there exist a subgroup $G$ of $\operatorname{Iso}(M)$ that acts transitively on $M$. From Remark \ref{rem:gx} one has that if $y=f(x)$ for an isometry $f\in G$, then $\tilde{\got p}^y=f_*(\tilde{\got p}^x)$ and so $\got s_y=f_*(\got s_x)$. In particular, $i_{\got s}(x)=i_{\got s}(y)$. So $i_{\got s}(M)$ is well defined and the assignment 
\begin{equation*}
    \got s:q\mapsto \got s_q
\end{equation*}
defines a $G$-invariant (hence $C^{\infty}$) distribution on $M$, called the \emph{distribution of symmetry} of $M$. Observe that if $\got{s}_x$ is non-degenerate for some $x\in M$, then $\got s$ is a non-degenerate distribution on $M$ and from Theorem \ref{teo:locsym}, $\got s$ is integrable and its leaves $L(q)$  are geodesically complete, homogeneous, totally geodesic, locally symmetric submanifolds of $M$. 

\section{The distribution of symmetry of a $2$-step nilpotent Lorentzian naturally reductive Lie group}\label{section6}

For a Riemannian simply connected, irreducible, compact normal homogeneous space $M = G / H$, which is not a symmetric space, the distribution of symmetry coincides with the distribution of fixed points of the (connected) isotropy representation \cite{olmos2013}. This was also proved for Riemannian naturally reductive nilpotent Lie groups \cite{reggiani2019}. In this section, we prove a similar result for Lorentzian $2$-step nilpotent Lie groups.

Let $M$ be a $G$-homogeneous manifold with $G \subset \Iso(M)$. For each $q \in M$, let $H^q = G_q$ be the isotropy subgroup at $q$. The \emph{isotropy representation} at of $M$ at $q$ is the faithful representation 
\begin{equation*}
    \rho^q: H^q\to \OO(T_q M), \qquad h\mapsto \rho^q(h)=dh_q.
\end{equation*}

Assume now that $M$ is simply connected and let $G_0$ be the connected component of the identity of $G$. Then $M$ is a $G_0$-homogeneous manifold and the isotropy $H^q_0:=(G_0)_q$ is connected. The \emph{connected isotropy representation}  is the representation 
\begin{equation*}
    \rho^q_0 = \rho^q|_{H^q_0} : H^q_0 \to \SO(T_qM).
\end{equation*}
From Remark \ref{rem:gx} one has that  $\rho^q(h)(\got s_q)=\got s_q$ for each $h\in H^q$. Let $\mathcal F^q$ (resp. $\mathcal F^q_0$) be the subspace of $T_qM$ given by the fixed points of $\rho^q$ (resp. $\rho^q_0$), i.e.,
\begin{align*}
  \mathcal F^q & =\{v \in T_qM : \rho^q(h)(v) = v, \text{ for all } h \in H^q\}, \\
  \mathcal F_0^q & =\{v \in T_qM : \rho_0^q(h)(v) = v, \text{ for all } h \in H^q_0\}.
\end{align*}
One can see that the assignment $q\mapsto \mathcal F^q$ (resp.\ $q\mapsto \mathcal F^q_0$) is a $C^{\infty}$ $G$-invariant (resp.\ $G_0$-invariant) distribution. 

\begin{theorem}\label{teo:maintheorem}
  Let $N$ be a simply connected $2$-step nilpotent Lie group endowed with a left-invariant Lorentzian metric, with non-degenerate center. Assume that the metric is naturally reductive with respect to the full isometry group and that the full isotropy $H$ satisfies $H = H^{\operatorname{aut}}$. Assume further that the representation $j$ defined in \eqref{eq:def_map_j} is injective.  Then the distribution of symmetry of $N$ is non-degenerate and coincides with the $\Iso_0(N)$-invariant distribution $\mathcal F_0$ determined by the fixed vectors of the connected isotropy representation of $N$. 
\end{theorem}

In order to prove this theorem we need some technical results. Let $N$ be a (non necessarily nilpotent) Lie group with a left-invariant metric. Then $N$ can be thought of a subgroup of $\Iso(N)$ via the monomorphism $L : N \to \Iso(N)$, $g\mapsto L_g$, and hence $\got n$ is a Lie subalgebra of $\got{iso}(N)$. Let $\Phi:\got{iso}(N)\to \mathcal K_c(N)$ be the Lie algebra anti-isomorphism defined by \eqref{eq:isokilling}. For $U \in \got n$ we denote by $U^*=\Phi(U)$ the corresponding right-invariant Killing vector field.

Recall that the Koszul form in left-invariant fields becomes
\begin{equation*}
  2\langle \nabla_U V, W\rangle = \langle [U, V], W \rangle - \langle [U, W], V \rangle - \langle [V , W], U \rangle, \qquad U, V, W \in \got n.
\end{equation*}
On the other hand, since $\mathcal L_{K^*}\langle\cdot,\cdot\rangle=0$ for each $K\in \got{iso}(M)$, from the Koszul formula we get that 
\begin{equation}\label{eq:nablakilling}
  2\langle \nabla_{U^*} V^*, W^*\rangle = \langle [U^*, V^*], W^* \rangle + \langle [U^*, W^*], V^* \rangle + \langle [V^*, W^*], U^* \rangle.
\end{equation}
Since $\Phi$ is a Lie algebra anti-isomorphism,  we get 
\begin{align*}
  2  \langle \nabla_{U^*} V^*, W^* \rangle & = -\langle  [U, V]^*, W^* \rangle - \langle [U, W]^*, V \rangle -\langle [V, W]^*, U^* \rangle \\
                    &= \langle [U,V]^*, W^* \rangle - \langle [U, W]^*, V \rangle -\langle [V, W]^*, U^* \rangle - 2\langle [U, V]^*, W^* \rangle.
\end{align*}
Since $U^*_e = U_e$, $V^*_e = V_e$ and $W^*_e = W_e$, we have
\begin{equation*}
    \langle(\nabla_{U^*} V^*)_e, W_e\rangle = \langle (\nabla_U V)_e, W_e \rangle - \langle [U, V]_e, W_e \rangle.
\end{equation*}
We conclude that 
\begin{equation}\label{eq:nablaizqder}
  (\nabla_{U^*} V^*)_e = (\nabla_U V )_e - [U, V]_e = (\nabla_U V)_e + [U^*, V^*]_e.
\end{equation}

\begin{lemma}\label{lemma:nabla}
  Let $(\got g,\got v,\pi)$ be a Lorentzian data set and let $N = N(\got g, \got v,\pi)$ be the associated simply connected Lie group and $\got{n} = \got{g} \oplus \got{v}$ its Lie algebra. For each $U \in \got n$ let $U^*$ be the right-invariant Killing vector field defined by $U$ (i.e. $U^* = \Phi(U)$). Let $X,Y\in \got v$, $Z,Z'\in \got g$. Then
  \begin{enumerate}
    \item $(\nabla_{X^{\ast}} Y^{\ast})_e =  \frac{1}{2}[X^{\ast}, Y^{\ast}]_e = -\frac{1}{2}[X,Y]_e$;
    \item $(\nabla_{X^{\ast}} Z^{\ast})_e = (\nabla_{Z^{\ast}} X^{\ast})_e = - (\frac{1}{2}\pi(Z)X)_e$;
    \item $(\nabla_{Z^{\ast}}Z'^{\ast})_e = 0$.
  \end{enumerate}
\end{lemma}

\begin{proof}
  From \eqref{eq:nabla_in_2stepnilpo} we have that $\nabla_{X}Y=\frac{1}{2}[X,Y]$. Then from \eqref{eq:nablaizqder}, 
  \begin{equation*}
      (\nabla_{X^*}Y^*)_e = \frac{1}{2}[X,Y]_e-[X,Y]_e = -\frac{1}{2}[X,Y]_e = -\frac{1}{2}[X,Y]^*_e = \frac{1}{2}[X^*,Y^*]_e.
  \end{equation*}
  Since $\got g = \got{z}(\got n)$, then $[Z,X] = [Z,Z'] = 0$ and then from \eqref{eq:nabla_in_2stepnilpo} and \eqref{eq:nablaizqder},
  \begin{equation*}
      (\nabla_{X^*}Z^*)_e = (\nabla_{X}Z)_e=(\nabla_{Z}X)_e = (\nabla_{Z^*}X^*)_e = -\left(\frac{1}{2}\pi(Z)X\right)_e
  \end{equation*}
  and $(\nabla_{Z^*}Z'^*)_e = (\nabla_{Z}Z')_e=0$.
\end{proof}

\begin{lemma}\label{lemma:lemaaux} 
  Let $(\got g,\got v,\pi)$ be a Lorentzian data set and suppose that $N = N(\got g, \got v,\pi)$ verifies the hypothesis of Theorem \ref{teo:maintheorem}. Decompose $\got g=\got c\oplus \overline{\got g}$, with $\overline{\got g}=[\got g,\got g]$. Let $\got s$ be the distribution of symmetry of $N$. Then $\got s_e=\got c$.  In particular, $\got s$ is non-degenerate and it coincides with the left-invariant distribution on $N$ defined by $\got c$.
\end{lemma}

\begin{proof} 
  Under the hypothesis of Theorem \ref{teo:maintheorem}, 
  \begin{equation*}
      \Iso(N) = \Iso^{\operatorname{aut}}(N) \simeq N \rtimes H
  \end{equation*}
  (cf.~ \cite[Proposition 3]{delbarco2014}) and so $\got{iso}(N)\simeq \got n\rtimes \got h$. Keeping the notations we have used so far, during the proof we shall denote by $\tilde U = \Phi(U)$ for a generic $U\in \got{iso}(\got n)$ and by $U^* = \Phi(U)$ the right invariant Killing vector field defined by an element $U\in \got n \subset \got{iso}(\got n)$, where $\Phi$ is the anti-isomorphism defined by \eqref{eq:isokilling}. 

  Let $v \in \got s_e \subset T_eN$ and let $\tilde{V} \in \mathcal K_c(N)$ be a (complete) transvection such that $\tilde{V}_e = v$. Then $\tilde{V} = \Phi(V)$ for some $V\in \got{iso}(N)$ and $V = U + D$, with $U\in \got n$ and $D\in \got h$. Hence we can decompose 
  \begin{equation*}
      \tilde{V} = U^{\ast} + \tilde{D}.
  \end{equation*}
  Recall that $\tilde{D}_e = 0$. 

  According to the decompositions
  \begin{equation*}
    \got n = \got g \oplus \got v, \ \ \got g = \overline{\got g } \oplus \got c, \ \ \got h = \overline{\got g} \oplus \got u,
  \end{equation*}
  we can write
  \begin{equation*}
      U = Z_{\overline{\got g}} + Z_{\got c} + X_{\got v}, \qquad D = D_{\overline{\got g}} + D_{\got u},
  \end{equation*}
  and so
  \begin{equation}
    \tilde{V} = Z^{\ast}_{\overline{\got g}} + Z^{\ast}_{\got c} + X^{\ast}_{\got v} + \tilde{D}_{\overline{\got g}} + \tilde{D}_{\got u}.
  \end{equation}

  Let $Z\in \got g$. Then 
  \begin{equation*}\label{eq:auxtrans}
    0 = (\nabla_{Z^{\ast}} \tilde{V})_e = (\nabla_{Z^{\ast}} U^{*})_e + (\nabla_{Z^{\ast}} \tilde{D})_e.
  \end{equation*}
  From Lemma \ref{lemma:nabla} we have that 
  \begin{equation*}
      (\nabla_{Z^{\ast}} U^*)_e = (\nabla_{Z^{\ast}} (Z^*_{\overline{\got g}} + Z^*_{\got c}))_e + (\nabla_{Z^{\ast}} X^*_{\got v})_e = 0 - \frac{1}{2}\left( \pi(Z) X_{\got v} \right)_e.
  \end{equation*}

  Let now $W\in \got n$. Then from \eqref{eq:nablakilling}, and since $\tilde{D}_e=0$, 
  \begin{align*}
    2\langle(\nabla_{Z^{\ast}} \tilde{D})_e, W^*_e \rangle& = \langle[Z^*, \tilde{D}]_e, W^*_e\rangle + \langle [Z^*,W^*]_e, \tilde{D}_e\rangle + \langle [\tilde{D}, W^*]_e, Z^*_e \rangle\\
    &= \langle [Z^*, \tilde{D}]_e, W^*_e\rangle + \langle [\tilde{D}, W^*]_e, Z^*_e \rangle.
  \end{align*}
  Now, from \eqref{eq:corcheteiso}, $[Z^*,\tilde{D}]_e = -\Phi([Z, D])_e = \Phi(D(Z))_e = D(Z)_e$, and in the same way $[\tilde{D}, W^*]_e = -D(W)_e$. Recall that $D \in \got h=\mathrm{Der}(\got n)\cap \got{so}(\got n)$, so
  \begin{equation*}
      2\langle(\nabla_{Z^{\ast}} \tilde{D})_e, W_e \rangle = \langle D(Z)_e, W_e\rangle - \langle D(W)_e, Z_e \rangle = 2\langle D(Z)_e, W_e \rangle.
  \end{equation*}
  Since $W \in \got n$ is arbitrary, we conclude that $(\nabla_{Z^*} \tilde{D})_e = D(Z)_e = D_{\overline{\got g}}(Z)_e + D_{\got u}(Z)_e$. From Theorem \ref{teo:16}, for $Z\in \got g$, $D_{\got u}(Z)=0$ and $D_{\overline{\got g}}(Z)\in \got g$. So
  \begin{equation}\label{eq:nablaeng}
    (\nabla_{Z^*} \tilde{D})_e = -\frac{1}{2} (\pi(Z) X_{\got v})_e + D_{\overline{\got g}}(Z)_e
  \end{equation}
  and then 
  \begin{equation*}
      -\frac{1}{2} (\pi(Z) X_{\got v})_e + D_{\overline{\got g}}(Z)_e = 0.
  \end{equation*}
  But $\pi(Z) X_{\got v} \in \got v$ and $D_{\overline{\got g}}(Z) \in \got g$, then  we must have $\pi(Z) X_{\got v} = 0$ and $D_{\overline{\got g}}(Z) = 0$. Since $Z \in \got g$ is arbitrary we conclude that $D_{\overline{\got g}} = 0$ and $X_{\got v} \in \cap_{Z'\in \got g}\pi(Z') = \{0\}$ and so $X_{\got v}^* = 0 $ and $\tilde{D}_{\overline{\got g}} = 0$.
  Therefore
  \begin{equation*}
    \tilde{V} = Z^{\ast}_{\overline{\got g}} + Z^{\ast}_{\got c} + \tilde{D}_{\got u}.
  \end{equation*}

  Let now $X \in \got v$. Then, with the same argument as before, we have that 
  \begin{equation}\label{eq:nablaenv}
    \begin{array}{rl}
     (\nabla_{X^{\ast}} \tilde{V})_e   & = (\nabla_{X^{\ast}} Z^{\ast}_{\overline{\got g}})_e + (\nabla_{X^{\ast}} Z^{\ast}_{\got c})_e + (\nabla_{X^{\ast}} \tilde{D}_{\got u})_e \\
                                       &= - \frac{1}{2} (\pi(Z_{\overline{\got g}}) X)_e - \frac{1}{2} (\pi(Z_{\got c}) X)_e  + \tilde{D}_{\got u}(X)_e. \\
    \end{array}
  \end{equation}
  It follows that $\pi(Z_{\overline{\got g}}) = -\pi(Z_{\got c}) + 2 D_{\got u}$.  Take an arbitrary $Z\in \got g$. Then $[\pi(Z_c), \pi(Z)] = \pi([Z_c, Z]_{\got g})=0$ and from Theorem \ref{teo:16}, $[\tilde{D}_{\got u}, \pi(Z)]=0$. Then $[Z_{\overline{g}}, Z]_{\got g}=0$, and since $\got g$ is semisimple we must have $Z_{\overline{\got g}}=0$.

  We conclude that 
  \begin{equation*}
      \tilde{V} = Z^{\ast}_{\got c} + \tilde{D}_{\got u}
  \end{equation*}
  and hence $v = (Z_{\got c})_e$ with $Z_{\got c} \in \got c$. 

  Now, let $Z\in \got c$. Observe that $\pi(Z) \in \got{so}(\got v)$ and for each $Z' \in \got g$, $[\pi(Z'), \pi(Z)]=\pi([Z', Z]_{\got g})=0$. Then from Theorem \ref{teo:maintheorem} we get that $\tilde{D}_{\got u} = \frac{1}{2}\pi(Z) \in \got u$. Let $\tilde{V} = Z^* + \tilde{D}_{\got u}$. Then $\tilde{V}_e = Z_e$ and from equations \eqref{eq:nablaeng} and \eqref{eq:nablaenv} it follows that $(\nabla \tilde{V})_e=0$. Hence $Z_e\in \got s_e$. 
\end{proof}

\begin{proof}[Proof of Theorem \ref{teo:maintheorem}]
  From Lemma \ref{lemma:lemaaux}, we have that $\got s_e = \{ Z_e : Z\in \got c \}$.

  Observe that since $N$ is connected, $\Iso(N)_0 = N\rtimes H_0$, where $H_0$ is described in Corollary \ref{cor:H0}. Hence, via the identification of $\got n$ with a subalgebra of $\got{iso}(N)$, $\got n$ is $\Ad(H_0)$-invariant and it is standard to see that $\rho^e_0(h)(V_e) = \Ad(h)(V)_e$ for each $h\in H_0$ and each $V\in \got n$. Since $H_0$ is connected, it follows that $V_e \in \mathcal F^e_0$ if and only if $[D, V]_{\got{iso}(N)} = D(V) = 0$ for all $D\in \got h$.

  Suppose $V = Z_{\got c} + Z_{\overline{\got g}} + X_{\got v}\in \got n$, where $Z_{\got c}\in \got c$, $Z_{\overline{\got g}}\in \overline{\got g}$ and $X_{\got v}\in \got v$. If $D\in \got h$, from Theorem \ref{teo:16}, $D = (\ad_{\got g}(Z), \pi(Z) + B)\in \got{so}(\got g)\times \got{so}(\got v)$ with $Z\in \overline{\got g}$ and $B\in \got u$. So
  \begin{equation}\label{eq:corcheteison}
    [D,V]_{\got{iso}(N)} = [Z, Z_{\overline{\got g}}]_{\got g} + \pi(Z)(X_{\got v}) + B(X_{\got v}).
  \end{equation}
  It follows immediately that if $V = Z_{\got c}\in \got c$, then $[D, V]_{\got{iso}(N)} = 0$ for each $D\in \got h$. Hence $\{ Z_e : Z\in \got c \} \subset \mathcal F^e_0$.

  On the other hand, if $V$ is such that $[D, V]_{\got{iso}(N)}=0$ for each $D\in \got h$ then, in particular, $[Z, V]_{\got{iso}(N)}=0$ for each $Z\in \overline{\got g}$ and $[B, V]_{\got{iso}(N)}=0$ for each $B\in \got u$. Taking $B=0$ in \eqref{eq:corcheteison}, we have that $[Z, Z_{\overline{\got g}}]_{\got g}=0$ and $\pi(Z)(X_{\got v})=0$ for each $Z\in \overline{\got g}$. Hence $X_{\overline{\got g}}=0$, and $Z_{\got v}\in \cap_{Z\in \overline{\got g}}\ker(\pi(Z)).$

  Taking $Z = 0$ in \eqref{eq:corcheteison}, we have that $B(X_{\got v})=0$ for each $B\in \got u$. Since $\pi(Z)\in \got u$ for each $Z\in \got c$, it follows that 
  \begin{equation*}
      X_{\got v}\in \bigcap_{Z\in \got g}\ker(\pi(Z)).
  \end{equation*}
  Then $X_{\got v}=0$ and so $V = Z_{\got c}\in \got c$.
\end{proof}

\bibliographystyle{amsalpha}  
\bibliography{references}  

\providecommand{\bysame}{\leavevmode\hbox to3em{\hrulefill}\thinspace}
\providecommand{\MR}{\relax\ifhmode\unskip\space\fi MR }
\providecommand{\MRhref}[2]{%
  \href{http://www.ams.org/mathscinet-getitem?mr=#1}{#2}
}
\providecommand{\href}[2]{#2}
\begin{thebibliography}{DSOV22}

\bibitem[BOR17]{berndt-olmos-reggiani2016}
J.~Berndt, C.~Olmos, and S.~Reggiani, \emph{{Compact homogeneous Riemannian manifolds with low coindex of symmetry}}, J. Eur. Math. Soc. \textbf{19} (2017), no.~1, 221--254.

\bibitem[CCR25]{cardoso2024}
I.~Cardoso, A.~Cosgaya, and S.~Reggiani, \emph{The moduli space of left-invariant metrics on six-dimensional characteristically solvable nilmanifolds}, Math. Nachr. \textbf{298} (2025), 1496--1520.

\bibitem[CdBR24]{conti2024}
D.~Conti, V.~del Barco, and F.~Rossi, \emph{On uniqueness of ad-invariant metrics}, Tohoku Math. J. \textbf{76} (2024), no.~3, 317--359.

\bibitem[CWZ22]{chen-wolf-zhang2020}
Z.~Chen, J.~A. Wolf, and S.~Zhang, \emph{{On the geodesic orbit property for Lorentz manifolds}}, J. Geom. Anal. \textbf{32} (2022), no.~3, 1--14.

\bibitem[dBO14]{delbarco2014}
V.~del Barco and G.~P. Ovando, \emph{{Isometric actions on pseudo-Riemannian nilmanifolds}}, Ann. Global Anal. Geom. \textbf{45} (2014), no.~2, 95--110.

\bibitem[DSO01]{diScala-olmos2001}
A.~J. Di~Scala and C.~Olmos, \emph{{The geometry of homogeneous submanifolds of hyperbolic space}}, Math. Z. \textbf{237} (2001), 199--209.

\bibitem[DSOV22]{DOV}
A.~J. Di~Scala, C.~Olmos, and F.~Vittone, \emph{{Homogeneous Riemannian manifolds with non trivial nullity}}, Transform. Groups \textbf{27} (2022), 31--72.

\bibitem[Ebe94]{eberlein1994}
P.~Eberlein, \emph{{Geometry of $2$-step nilpotent groups with a left invariant metric}}, Ann. Sci. École Norm. Sup. \textbf{27} (1994), no.~5, 611--660.

\bibitem[Ebe96]{eberlein1996}
\bysame, \emph{Geometry of nonpositively curved manifolds}, Chicago Lectures in Mathematics, Chicago Press, 1996.

\bibitem[Gor85]{gordon1985}
C.~S. Gordon, \emph{{Naturally reductive homogeneous Riemannian manifolds}}, Can. J. Math. \textbf{37} (1985), no.~3, 467--487.

\bibitem[JSC10]{javaloyes-sanchez2010}
M.~A. Javaloyes and M.~Sánchez~Caja, \emph{{An introduction to Lorentzian geometry and its applications}}, XVI Escola de Geometria Diferencial, IME, USP, 2010.

\bibitem[Kap81]{kaplan1981}
A.~Kaplan, \emph{{Riemannian nilmanifolds attached to Clifford modules}}, Geom. Dedicata \textbf{11} (1981), no.~2, 127--136.

\bibitem[Lau98]{lauret1998}
J.~Lauret, \emph{{Naturally reductive homogeneous structures on $2$-step nilpotent Lie groups}}, Rev. Unión Mat. Argent. \textbf{41} (1998), no.~2, 15--24.

\bibitem[Lau99]{lauret1999}
\bysame, \emph{{Homogeneous nilmanifolds attached to representations of compact Lie groups}}, Manuscripta Math. \textbf{99} (1999), 287--309.

\bibitem[May21]{may2021}
R.~May, \emph{{The index of symmetry for a left-invariant metric on a solvable three-dimensional Lie group}}, 2021, \url{https://arxiv.org/abs/2103.15789}.

\bibitem[NW23]{nikolayevsky-wolf2022}
Y.~Nikolayevsky and J.~A. Wolf, \emph{{The structure of geodesic orbit Lorentz nilmanifolds}}, J. Geom. Anal. \textbf{33} (2023), no.~3, 1--12.

\bibitem[ORT14]{olmos2013}
C.~Olmos, S.~Reggiani, and H.~Tamaru, \emph{{The index of symmetry of compact naturally reductive spaces}}, Math. Z. \textbf{277} (2014), 611--628.

\bibitem[OS95]{olmos95}
C.~Olmos and M.~Salvai, \emph{{Holonomy of homogeneous vector bundles and polar representations}}, Indiana Univ. Math. J. \textbf{44} (1995), no.~3, 1007--1015.

\bibitem[Ova13]{ovando2010}
G.~P. Ovando, \emph{{Naturally reductive pseudo-Riemannian $2$-step nilpotent Lie groups}}, Houston J. Math. \textbf{39} (2013), no.~1, 147--167.

\bibitem[Pod15]{podesta2015}
F.~Podestà, \emph{{The index of symmetry of a flag manifold}}, Rev. Mat. Iberoam. \textbf{31} (2015), no.~4, 1415--1422.

\bibitem[Reg18]{reggiani2018}
S.~Reggiani, \emph{{The index of symmetry of three-dimensional Lie groups with a left-invariant metric}}, Adv. Geom. \textbf{18} (2018), no.~4, 395--404.

\bibitem[Reg19]{reggiani2019}
\bysame, \emph{{The distribution of symmetry of a naturally reductive nilpotent Lie group}}, Geom. Dedicata \textbf{200} (2019), no.~1, 61--65.

\bibitem[Reg21]{reggiani2021}
\bysame, \emph{Manifolds admitting a metric with co-index of symmetry $4$}, Manuscripta Math. \textbf{164} (2021), no.~3, 543--553.

\bibitem[Wil82]{wilson1982}
E.~N. Wilson, \emph{{Isometry groups on homogeneous nilmanifolds}}, Geom. Dedicata \textbf{12} (1982), no.~3, 337--346.

\bibitem[Wol62]{wolf1962}
J.~A. Wolf, \emph{{On locally symmetric spaces of non-negative curvature and certain other locally homogeneous spaces}}, Comment. Math. Helv. \textbf{37} (1962), no.~1, 266--295.

\bibitem[Wu64]{wu1964}
H.~Wu, \emph{{On the de Rham decomposition theorem}}, Illinois J. Math. \textbf{8} (1964), no.~4, 291--311.

\end{thebibliography}

\end{document}